\documentclass[12pt]{amsart}
\usepackage[utf8]{inputenc}
\usepackage{setspace}
\usepackage[english]{babel}
\usepackage{mathabx}
\usepackage{framed}
\usepackage{soul}
\usepackage[all]{xy}
\usepackage[margin=1in]{geometry}
\usepackage{arydshln}
\usepackage[colorinlistoftodos]{todonotes}
\usepackage{pb-diagram}
\usepackage[mathscr]{euscript}
\usepackage{multicol}
\usepackage{stmaryrd} 
\usepackage{empheq} 
\usepackage{tikz-cd}
\tikzset{
  symbol/.style={
    draw=none,
    every to/.append style={
      edge node={node [sloped, allow upside down, auto=false]{$#1$}}}
  }
}

\usepackage{amsmath,amsfonts,amssymb,amsthm,mathrsfs,latexsym,mathtools}
\usepackage{commath}
\usepackage[T1]{fontenc}
\usepackage{hyperref}
\usepackage{caption} 
\usepackage{subcaption} 

\usepackage{lipsum}

\DeclareMathAlphabet{\mathbbmsl}{U}{bbm}{m}{sl}

\usetikzlibrary{matrix}
\title{Extension of the Topological Abel-Jacobi Map for Cubic Threefolds}

\newcommand{\C}{\mathbb{C}} 
\newcommand{\Z}{\mathbb{Z}}

\newcommand{\Hv}{H_{\textup{van}}}

\newcommand{\prim}{\textup{prim}}
\newcommand{\van}{\textup{van}}
\newcommand{\Bl}{\textup{Bl}}

\newcommand{\Aut}{\textup{Aut}}
\newcommand{\Sym}{\textup{Sym}}

\def\acts{\curvearrowright} 

\newtheorem{theorem}{Theorem}[section]
\newtheorem{definition}[theorem]{Definition}
\newtheorem{proposition}[theorem]{Proposition}
\newtheorem{corollary}[theorem]{Corollary}
\newtheorem{question}[theorem]{Question}
\newtheorem{example}[theorem]{Example}
\newtheorem{lemma}[theorem]{Lemma}
\newtheorem{conjecture}[theorem]{Conjecture}

\newtheorem{remark}[theorem]{Remark}

\newtheorem{thm}{Theorem}

\author{Yilong Zhang}
\address{Department of Mathematics\\
Purdue University\\
  150 N University St, West Lafayette, IN 47907}
\email{zhan4740@purdue.edu}

\date{Jan 3, 2024}

\begin{document}
\maketitle
\begin{abstract}

    The difference $[L_1]-[L_2]$ of a pair of skew lines on a cubic threefold defines a vanishing cycle on the cubic surface as the hyperplane section spanned by the two lines. By deforming the hyperplane, the flat translation of such vanishing cycle forms a 72-to-1 covering space $\mathcal{T}_v\to U$ of a Zariski open subspace of $(\mathbb P^4)^*$. Based on a lemma of Stein on the compactification of finite analytic covers, we found a compactification of $\mathcal{T}_v$ to which the topological Abel-Jacobi map $\mathcal{T}_v\to J(X)$ extends. Moreover, the boundary points of the compactification can be interpreted in terms of local monodromy and the singularities on cubic surfaces. We prove the associated map on fundamental groups of topological Abel-Jacobi map is surjective.
\end{abstract}
\tableofcontents

\section{Introduction}
Let $X\subseteq \mathbb P^N$ be a smooth projective variety of dimension $2n-1$ over $\C$ and $Y$ be a smooth hyperplane section. The vanishing cohomology of $Y$ is the kernel of the Gysin morphism 
\begin{equation}\label{eqn_Intro_Hv}
    H^{2n-2}_{\van}(Y,\mathbb Z)=\ker(H^{2n-2}(Y,\mathbb Z)\to H^{2n}(X,\mathbb Z)).
\end{equation}

When $Y$ varies in the universal family of smooth hyperplane sections, the vanishing cohomology forms a $\Z$-local system $\mathcal{H}^{2n-2}_{\van}$ over an open subspace $U$ of $(\mathbb P^N)^*$. The \'etale space of $\mathcal{H}^{2n-2}_{\van}$ is naturally an analytic covering space $\mathcal{T}\to U$. There is a closed subspace $\textup{Hdg}(\mathcal{T})\subseteq \mathcal{T}$, called the \textit{locus of Hodge classes}, parameterizing Hodge classes of the middle dimension on $Y$. According to Cattani, Deligne, and Kaplan \cite{CDK}, each connected component of $\textup{Hdg}(\mathcal{T})$ is algebraic.

Schnell defined an analytic compactification of $\textup{Hdg}(\mathcal{T})$ by first obtaining a normal analytic compactification $\bar{\mathcal{T}}$ of $\mathcal{T}$ using Hodge module theory and then taking analytic closure of $\textup{Hdg}(\mathcal{T})$ (cf. \cite{SchExtLocus} and Appendix \ref{App_Schnell}). The resulting analytic space $\overline{\textup{Hdg}(\mathcal{T})}$ is, in fact, algebraic, and captures the meaning of Hodge classes "in the limit" on the boundary points. The initial motivation of this research is to understand these boundary points.

\begin{question}\label{Question_1}
What is the geometric meaning of the boundary points of $\overline{\textup{Hdg}(\mathcal{T})}$? 
\end{question}

On the other hand, there is a distinguished component $\mathcal{T}_v$ of $\mathcal{T}$ containing a \textit{primitive vanishing cycles}, i.e., a class which is monodromy conjugate to the class of a vanishing sphere of nodal degeneration (cf. Proposition \ref{Prop_Tv}). According to a conjecture of Herb Clemens (cf. Conjecture \ref{Conj_Clemens_Tube}), the component $\mathcal{T}_v$ captures enough topological information of $X$. So we're particularly interested in understanding the component $\bar{\mathcal{T}}_v$ which compactifies $\mathcal{T}_v$. 

There is a real analytic map called the \textit{topological Abel-Jacobi map} (cf. \cite{YZ_TAJ} and \cite{Zhao})
\begin{equation}\label{Intro_eqn_TAJ}
  \Psi_{\textup{top}}: \mathcal{T}\to J_{\prim}(X), 
\end{equation}
generalizing Griffiths' Abel-Jacobi map and its restriction to $\text{Hdg}(T)$ is holomorphic. Here $J_{\prim}(X)$ is the primitive intermediate Jacobian of $X$ in the middle dimension. We ask

\begin{question}\label{Question_TAJext}
    Does the topological Abel-Jacobi map \eqref{Intro_eqn_TAJ} extend to Schnell's compactification of $\mathcal{T}_v$?
\end{question}
Note Schnell's compactification of $\mathcal{T}$ depends on the choice of completion of the base, so to answer Question \ref{Question_TAJext}, we should also specify compactification of $U$. Of course, we have a canonical choice: the projective space $(\mathbb P^N)^*$. To approach these questions, we start from $\dim (X)=3$, and when $X$ is a hypersurface in $\mathbb P^4$. The minimal degree for this question to be interesting is $3$: a cubic threefold. 
\subsection{Cubic Threefolds}
A smooth hyperplane section of a cubic threefold is a cubic surface denoted as $S$. A primitive vanishing cycle is class $\alpha\in H^2(S,\Z)$ such that
\begin{equation}\label{rootsystem_eqn}
    \alpha\cdot \alpha=-2, \alpha\cdot h=0,
\end{equation}
where $h$ is the hyperplane class. There are 72 such classes and correspond to the root system of Lie algebra $\mathbb E_6$. By varying the hyperplanes, the root system varies and form a 72-to-1 connected cover
\begin{equation}\label{Intro_eqn_pi_v}
    \pi_v:\mathcal{T}_v\to U,
\end{equation}
whose monodromy group is the Weyl group of $\mathbb E_6$.

Since any primitive vanishing cycle can be written as the difference $[L_1]-[L_2]$ of two skew lines, $\mathcal{T}_v$ already parameterizes (algebraic) Hodge classes, so Question \ref{Question_1} reduces to understand boundary points of compactification of $\mathcal{T}_v$. So the two questions can be combined into the following.

\begin{question}\label{Question_main}
    When $X$ is a cubic threefold, is there a geometric compactification of $\mathcal{T}_v$, with an understanding of boundary points, such that the topological Abel-Jacobi map \eqref{Intro_eqn_TAJ} extend to the compactification of $\mathcal{T}_v$?
\end{question}

Our first goal is to understand Schnell's compactification $\bar{\mathcal{T}}_v$. Since the monodromy of \eqref{Intro_eqn_pi_v} is finite, Schnell's compactification is a normal algebraic variety and coincides with an analytic compactification of Stein (cf. Lemma \ref{analytic_cover_closure}).

\subsection{Abel-Jacobi Map and Compactification of $\mathcal{T}_v$}
The intermediate Jacobian $J(X)$ of the cubic threefold $X$ is a principally polarized abelian variety of dimension five. The set of lines of $X$ is parameterized by a surface $F$ of general type.  The Abel-Jacobi map defined by Clemens and Griffiths \cite{CG} is a morphism 
\begin{equation}\label{Intro_eqn_AJ}
    \Psi:F\times F\to J(X),
\end{equation}

The image of \eqref{Intro_eqn_AJ} is the theta divisor $\Theta$. According to Beauville \cite{Beauville}, $\Theta$ has a unique triple point singularity $0$, and the blow-up $\Bl_0(\Theta)$ is smooth with the exceptional divisor isomorphic to $X$.

Two general lines $L_1, L_2$ on the cubic threefold $X$ are skew.  They span a hyperplane $H=\textup{Span}(L_1,L_2)$ and cuts out a smooth cubic surface $X_{H}$. Moreover, the class $[L_1]-[L_2]$ is a primitive vanishing cycle on $X_H$. Conversely, any primitive vanishing cycle arises from the difference of two skew lines exactly 6 times. By varying this construction in the family, we get a 6-to-1 lifting of the 72-to-1 cover \eqref{Intro_eqn_pi_v}

\begin{figure}[ht]
    \centering
\begin{tikzcd}
(F\times F)^{\circ}\arrow[r] \arrow[dr] & \mathcal{T}_v \arrow[d] \\
& U.
\end{tikzcd}
\end{figure}{}

$(F\times F)^{\circ}$ can be compactified in a double cover of the Hilbert scheme of a pair of skew lines of the cubic threefold $X$, and according to \cite{YZ_SkewLines}, such a double cover is isomorphic to the blow-up $\Bl_{\Delta_F}(F\times F)$ and dominates $\bar{\mathcal{T}}_v$. Using the structure of the Hilbert scheme and the result of Beauville on the extension of the Abel-Jacobi map to the blow-ups, we obtain



\begin{proposition} (cf. Theorem \ref{barT'BlowupThm})\label{IntroVCThm}
There is birational morphism $\textup{Bl}_0(\Theta)\to \bar{\mathcal{T}}_v$, which contracts finitely many elliptic curves corresponding to the Eckardt points on the cubic threefold $X$. 
\end{proposition}

Here an Eckardt point is a point $p\in X$ through which infinitely many lines on $X$ pass. The hyperplane section tangent to an Eckardt point $T_pX$ is a cone over an elliptic curve — which has an elliptic singularity at the cone point. For general cubic threefold $X$, there is no Eckardt point, so $\bar{\mathcal{T}}_v\cong \Bl_0(\Theta)$. 

Note that a primitive vanishing cycle on a cubic surface can be also written as $[C]-h$ where $C$ is a twisted cubic, and $h$ is a hyperplane class. So $\mathcal{T}_v$ (and therefore $\Theta$) is also dominated by an open subspace of the Hilbert scheme of twisted cubics. Therefore, $\Theta$ is parameterized by certain Gieseker stable moduli space of coherent sheaves. This is first considered By Beauville \cite{Bea00} and sharpened by Bayer et al. \cite{BBF+}. We discussed the relation to our work in \cite[Section 6]{YZ_SkewLines}.

\subsection{Limiting Primitive Vanishing Cycles and Resolution of ADE Singularities}

A fiber of $\mathcal{T}_v\to U$ corresponds to the $72$ primitive vanishing cycles (or roots) on a smooth cubic surface. The boundary points of $\bar{\mathcal{T}}_v$ capture the notion of primitive vanishing cycles "in the limit" as a general hyperplane section specialize to a singular hyperplane section. So we have the following definition.

\begin{definition}\label{Intro_Def_LPVS}\normalfont
Let $t_0\in (\mathbb P^4)^*\setminus U$. Call the set theoretic fiber $PV_{t_0}$ of $\bar{\mathcal{T}}_v\to (\mathbb P^4)^*$ at $t_0$ to be the set of \textit{limiting primitive vanishing cycles} on the cubic surface $X_{t_0}$.
\end{definition}

Assume $X_{t_0}$ has only ADE singularities, the universal hyperplane sections of $X$ around a hyperplane section $X_{t_0}$ captures "maximal" topological information in the sense that the local monodromy group of vanishing cohomology on the smooth fiber near $X_{t_0}$ is the same as the monodromy group in the semi-universal deformation of the ADE singularities on $X_{t_0}$ (cf. Proposition \ref{MonoActGloProp}).

Suppose $X_{0}=X_{t_0}$ has at worst ADE singularities, then the minimal resolution $\tilde{X}_0\to X_0$
has exceptional divisors union of a bunch of $(-2)$ curves determined by the Dynkin diagram of the corresponding ADE type of the singularities. On the other hand, these effective $(-2)$ curves generate a subgroup $W_e$ of the Weyl group $W(\mathbb E_6)$. Then we have the following interpretation of the limiting primitive vanishing cycles.

\begin{proposition} (cf. Theorem \ref{MonOrbThm}) \label{Intro_MonOrbThm}
 $PV_{t_0}$ is identified with the orbit of the group action $W_e\acts R(\mathbb E_6)$.
\end{proposition}

 The orbit space $W_e\acts R(\mathbb E_6)$ was originally defined in \cite{LLSvS} and is used to parameterize the reduced Hilbert scheme of generalized twisted cubics on $X_H$. We believe the set of such orbits is a natural notion of "root system" on a cubic surface with ADE singularities.

Now Proposition \ref{IntroVCThm} and \ref{Intro_MonOrbThm} provide an answer to Question \ref{Question_main} for a general cubic threefold.

\begin{theorem}\label{Intro_thm_generalX}
    When $X$ is general, we choose the compactification $\bar{U}=(\mathbb P^4)^*$ of the base. Then 

    \begin{itemize}
        \item[(1)] Stein's compactification $\bar{\mathcal{T}}_v$ is isomorphic to the blow-up of the theta divisor $\Bl_0(\Theta)$.
        \item[(2)] A fiber of $\bar{\mathcal{T}}_v\to (\mathbb P^4)^*$ at a point $t_0$ corresponds to the orbits of the subgroup of Weyl group $W_e$ generated by the $(-2)$ curves on the minimal resolution of the cubic surface $X_{t_0}$ with ADE singularities. 
        \item[(3)]  The topological Abel-Jacobi map extends to $\bar{\mathcal{T}}_v$.
    \end{itemize}

\end{theorem}

\subsection{Extension of Topological Abel Jacobi Map} \label{Intro_Sec_ExtTAJ}
 The Abel-Jacobi map \eqref{Intro_eqn_AJ} factors through the topological Abel-Jacobi map by restricting to an open subspace
\begin{equation}\label{Intro_eqn_TAJ_cubic3fold}
    (F\times F)^{\circ}\xrightarrow{} \mathcal{T}_v\xhookrightarrow{\Psi_{\textup{top}}} \Theta\subseteq J(X).
\end{equation}

When $X$ is general, via composition $\Bl_0(\Theta)\to \Theta\hookrightarrow J(X)$, the topological Abel-Jacobi map extends. However, when there is an Eckardt point, there is no such extension because via the same map, elliptic curves are sent isomorphically onto elliptic curves in $J(X)$, and on the other hand, elliptic curves are contracted to points on $\bar{\mathcal{T}}_v$  (cf. Proposition \ref{IntroVCThm}).

 So we look for an alternative compactification of $\mathcal{T}_v$ that carries geometric meaning for limiting primitive vanishing cycles on the Eckardt hyperplane section and to which the topological Abel-Jacobi map extends.

From the point of view of moduli space, the cubic surface with an Eckardt point is unstable \cite{ACT, Bea09}. However, we can always replace it with a semistable limit through a one-parameter family. This amounts to taking a base change followed by a birational modification. Using techniques in three-dimensional MMP, we prove

\begin{proposition} (cf. Corollary \ref{Cor_StableLim})
   The semistable limit of a general pencil of hyperplane sections through an Eckard hyperplane section is isomorphic to the cubic surface $S_{lim}$ which arises as the cyclic cover of $\mathbb P^2$ on an elliptic $E$.
\end{proposition}

The 27 lines on the cyclic cover arise as the pullback of 9 tangent lines to the flex points of $E$. The monodromy action permutes the three sheets and permutes the 27 lines. This allows us to determine the limiting primitive vanishing cycles through a one-parameter family.

Our new compactification is the following: We blow up points $T_pX\in (\mathbb P^4)^*$ that correspond to tangent spaces of Eckardt points and denote the new base as $\tilde{(\mathbb P^4)^*}$. It keeps track of one-dimensional families through the Eckardt cone. 

\begin{theorem}\label{Intro_Thm_TAJext}
   Let $X$ be a smooth cubic threefold $X$ (with Eckardt points). We choose the compactification $\tilde{(\mathbb P^4)^*}$ of $U$. 
    
    \begin{itemize}
 \item[(1)] The Stein compactification $\tilde{\mathcal{T}}_v$ is birational to $\Bl_0(\Theta)$.
    \item[(2)] The fiber $\tilde{\mathcal{T}}_v\to \tilde{(\mathbb P^4)^*}$ over a general point of an exceptional $\mathbb P^3\subseteq \tilde{(\mathbb P^4)^*}$ has cardinality $24$ and corresponds the cyclic $\Z_3$-action on the root system on the semistable limit $S_{lim}$ through a one parameter family. 
     \item[(3)] The topological Abel-Jacobi map extends to $\tilde{\mathcal{T}}_v\to \Bl_0J(X)$.
    \end{itemize}
\end{theorem}
This generalizes Theorem \ref{Intro_thm_generalX} to all smooth cubic threefolds. In particular, it provides a full answer to Question \ref{Question_main}.

\subsection{Tube Mapping}

To study the topology of the locus of the primitive vanishing cycles $\mathcal{T}_v$, one can consider the fundamental group induced from the topological Abel-Jacobi map
 \eqref{Intro_eqn_TAJ}
$$(\Psi_{\textup{top}})_*:\pi_1(\mathcal{T}_v,\alpha_0)\to \pi_1(J_{\prim},0)\cong H_{2n-1}(X,\Z)_{\prim}.$$

Equivalently, a loop in $\pi_1(J_{\prim},0)$ corresponds to a loop $l\in \pi_1(U,t_0)$ such that $l$ stabilizes $\alpha_0$ under monodromy action. The trace of the $\alpha_0$ along the loop forms a $(2n-1)$-cycle and defines a primitive homology class. Such a map is called \textit{tube mapping}. Herb Clemens conjectured that the image of the tube mapping has maximal rank.

\begin{conjecture}\label{Conj_Clemens_Tube} (Clemens)
The image of the tube mapping 
\begin{equation}\label{tube mapping}
    \{([l],\alpha_0)\: |\:[l]\in \pi_1(U,t),l_*\alpha_0=\alpha_0\}\to H_n(X,\mathbb Z)_{\textup{prim}}
\end{equation}
has maximal rank.

\end{conjecture}


 Schnell proved that when $\alpha_0$ runs over all classes in vanishing cohomology, the image tube mapping has maximal rank \cite{Tube}.

We verify Conjecture \ref{Conj_Clemens_Tube} for cubic threefolds.

\begin{proposition}(cf. Proposition \ref{Prop_TubeCubic3fold}) \label{Intro_Prop_Clemens_conj}
    The tube mapping for primitive vanishing cycles \eqref{tube mapping} is subjective for cubic 3-folds.
\end{proposition}

\subsection*{Outline}
In Section \ref{Section_Cubic3fold_Prelim}, we will review basic facts on cubic surfaces and cubic threefolds. In Section \ref{Section_Tvbar}, we study the 72-to-1 over $\mathcal{T}_v\to U$ and its various compactifications. In particular, Proposition \ref{IntroVCThm} will be proved. In section \ref{Section_Cubic3fold_Boundary}, we will relate the boundary points on $\bar{\mathcal{T}}_v$ to the Lie theory of the root system of the minimal resolution of the cubic surfaces with ADE singularities and prove Proposition \ref{Intro_MonOrbThm}. In Section \ref{Section_Cubic3fold_ExtAJ}, we'll study the extension of the topological Abel-Jacobi map $\mathcal{T}_v\to J(X)$ and prove Theorem \ref{Intro_Thm_TAJext}. In Appendix \ref{App_PVC}, we introduce the notion of primitive vanishing cycles and some basic properties. In Appendix \ref{App_Schnell}, we will review Schnell's compactification of \'etale space of a VHS and Stein's Lemma of compactification of finite analytic cover.

\subsection*{Acknowledgement} I would like to thank my advisor, Herb Clemens, for introducing me to this topic, answering my questions, and for his constant encouragement. Besides, I would like to thank Lisa Marquand, Kenji Matsuki, Wenbo Niu, Christian Schnell, Dennis Tseng, and Xiaolei Zhao for many useful communications. 

\section{Preliminaries}\label{Section_Cubic3fold_Prelim}

\subsection{Root System on Cubic Surfaces}
A cubic surface $S$ can be obtained by blowing up 6 points in general position on $\mathbb P^2$. So its second cohomology $H^2(S,\Z)$ is isomorphic to $\Z^7$ with the hyperplane class $h=3e_0-e_1-\cdots-e_6$, where $e_0$ is the class of the pullback of a general line on $\mathbb P^2$ and $e_1,\ldots, e_6$ be the classes of the exceptional divisors.

The vanishing cohomology $\Hv^2(S,\Z)$ is isomorphic to the orthogonal space $h^{\perp}$, which is also isomorphic to the $\mathbb E_6$-lattice with basis $\alpha_1=e_0-e_1-e_2-e_3$, $\alpha_i=e_{i-1}-e_i$, $i=2,\ldots, 6$. The intersection pairing $(\cdot,\cdot)$ on $h^{\perp}$ is given by the Cartan matrix, where $\alpha_i^2=-2$, and $\alpha_i\cdot\alpha_j=1$ if and only if the two roots $\alpha_i$ and $\alpha_j$ are adjacent in the Dynkin diagram, and otherwise zero. 

\begin{figure}[h]
         \centering
         \includegraphics[width=0.4\textwidth]{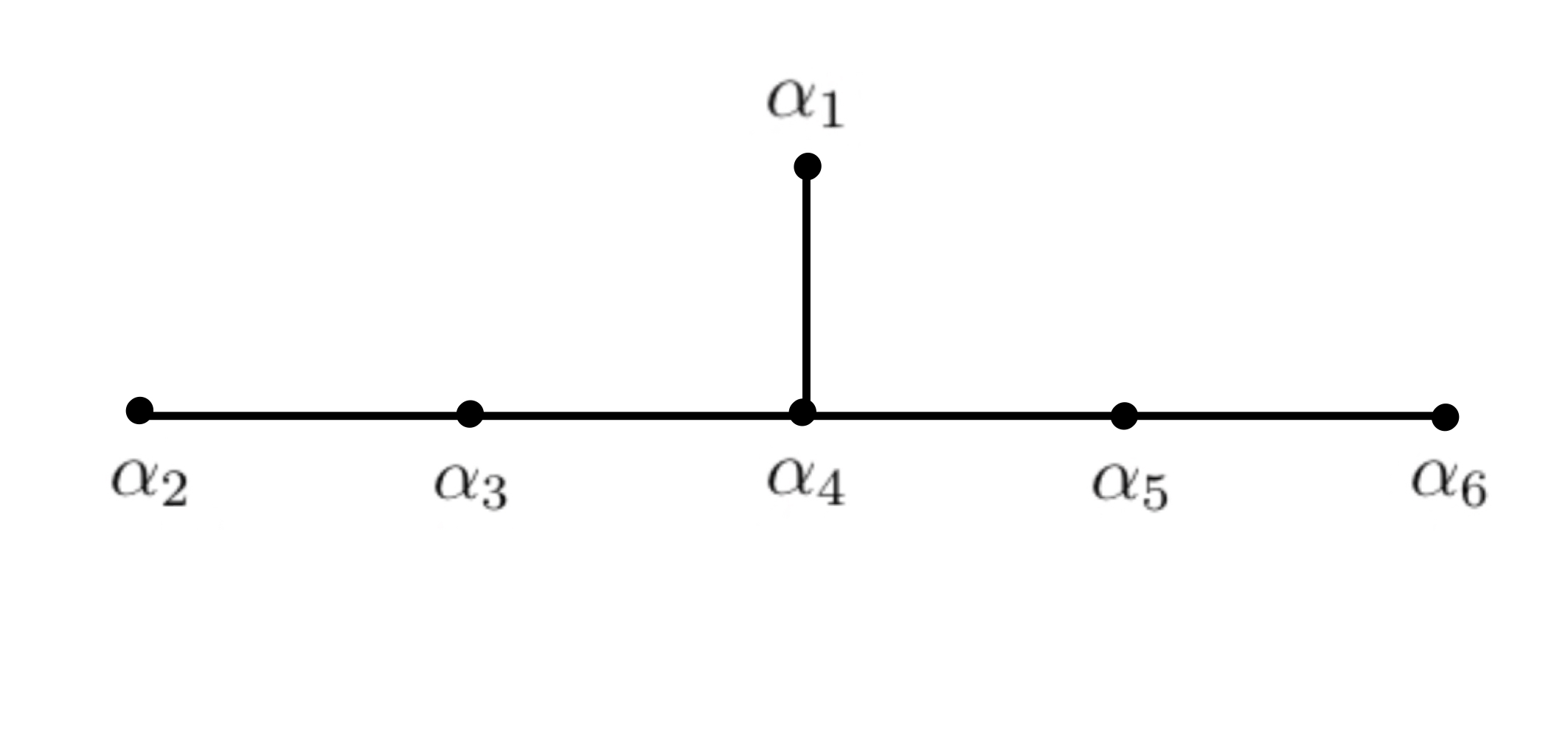}
         \caption{$\mathbb E_6$ Dynkin Diagram}
\end{figure}

\begin{definition}\label{Def_root}\normalfont
The \textit{root system} $R_S$ of the cubic surface $S$ is the set 
 \begin{equation}
     R_S=\{\alpha\in \Hv^2(S,\Z)|\alpha^2=-2\}.\label{72roots}
 \end{equation}
Call $\alpha\in R_S$ a root of $S$.
\end{definition}

The root system $R_S$ together with the intersection pairing $(\cdot,\cdot)$ is isomorphic to the root system of $\mathbb E_6$ up to a sign. In particular, it consists of 72 roots, and the automorphism of $R_S=(R_S,(\cdot,\cdot))$ is the Weyl group $W(\mathbb E_6)$, which is also the monodromy group when $S$ varies in the universal family of smooth cubic surfaces. One refers to \cite[Section 8.2]{Dolgachev} for the root system of $\mathbb E_6$, or \cite[Chapter III]{Hum} for general theory on root systems.

The Weyl group $W(\mathbb E_6)$ is generated by reflections $\{r_{\alpha}\}_{\alpha\in R_S}$, where
\begin{equation}\label{eqn_PLtransform}
    r_{\alpha}:\beta\mapsto \beta+(\beta,\alpha)\alpha
\end{equation}
is a reflection with respect to the hyperplane $H_{\alpha}=\{(\cdot,\alpha)=0\}$ associated to a root $\alpha$ and has the following geometric meaning: Consider a family of cubic surfaces $\{S_t\}_{t\in \Delta}$ parameterized by a holomorphic disk $\Delta$ such that $S_t$ is smooth when $t\neq 0$ and $S_0$ has an ordinary double point, there is a vanishing cycle $\alpha$ on nearby $S_{t_0}$. Take a loop $l$ whose class in $\pi_1(\Delta^*,t_0)$ is a generator, then the monodromy representation 
\begin{equation}\label{Eqn_PL}
    \rho_{\Delta}:\pi_1(\Delta^*,t_0)\to \Aut H^{2}(S_{t_0},\Z),
\end{equation}
is the same as \eqref{eqn_PLtransform}, where the root $\alpha\in S_t$ on the nearby smooth cubic surface is the vanishing cycle. This is called the \textit{Picard-Lefschetz transformation}.

\begin{example}\label{Exa_CubicNode}
\normalfont 
To describe the vanishing cycle of the nodal family $\{S_t\}_{t\in \Delta}$ geometrically, we regard $S_t$ as the blow-up of 6 general points $\{p_1(t),\ldots, p_6(t)\}$ on $\mathbb P^2$. When $t=0$, the 6 points lie on a conic $Q$, so the linear system $\mathfrak{o}$ of cubics through the 6 points induces an embedding $\textup{Bl}_{p_1(t),\ldots,p_6(t)}\mathbb P^2\to \mathbb P^3$ when $t\neq 0$ and contracts the strict transform of conic $Q$ to the node of $S_0$ when $t=0$.

Take a double cover of $\tilde{\Delta}\to \Delta$ branched at $0$ and base change by $t=s^2$, there is a commutative diagram

\begin{figure}[ht]
    \centering
\begin{tikzcd}
\tilde{\mathcal{S}}\arrow[d,"f"] \arrow[r,"h"] & \mathcal{S}\arrow[dl,"g"]\\
\tilde{\Delta}
\end{tikzcd}
\end{figure}
Here $g$ is the base change family and $f$ is submersion whose fiber is $\Bl_{p_1(s),\ldots,p_6(s)}\mathbb P^2$. $h$ is a small resolution (in the category of complex analytic manifolds) that is isomorphic over $s\neq 0$ and its restriction to $s=0$ is the minimal desingularization
$$\tilde{S}_0\to S_0$$
which contracts $\tilde{Q}$ to the node. Then the vanishing cycle on $S_s$ is the flat translation of the class $[\tilde{Q}]=2e_0-e_1-\cdots-e_6$ to the nearby smooth fiber.
\end{example}

\begin{proposition}\label{VC=[L]-[M]Prop}
Every root on $S$ can be written as the difference $[L_1]-[L_2]$ for a pair of skew lines $L_1,L_2$ in $S$ in exactly 6 different ways.
\end{proposition}
\begin{proof}

One can choose a planar representation of $S$ as a blow-up of 6 points on $\mathbb P^2$ corresponding to a given root $\alpha$, so $\alpha$ is expressed as 
\begin{equation}
    \alpha=2e_0-e_1-\cdots-e_6.\label{eqn_cubicsurface_1}
\end{equation}
On the other hand, $[G_i]=2e_0-\sum_{j\neq i}e_j$ and $E_i=e_i$ are disjoint lines (i.e., $(-1)$ classes of degree one), one has $\alpha=[G_i]-[E_i]$ for each $i=1,\ldots,6$. It's direct to check these are the only ways to express $\alpha$ as the class of difference of two lines.
\end{proof}

\subsection{Locus of Primitive Vanishing Cycles}
We want a version of Proposition \ref{VC=[L]-[M]Prop} in families. Let $X$ be a smooth cubic threefold, and $U\subseteq (\mathbb P^4)^*$ parameterizes smooth hyperplane sections of $X$. Recall $T$ is the \'etale space of the local system of vanishing cohomology $\mathcal{H}_{\textup{van}}^2$ over $U$. 

Let $\mathcal{T}_v=\{(\alpha_t,t)\in T|\alpha_t^2=-2,\  \alpha_t\cdot h=0\}$. Then the natural projection 
\begin{equation}\label{Eqn_LocusPV_CT}
   \pi_v: \mathcal{T}_v\to U
\end{equation}
is a covering space of degree 72 whose fiber at $t$ is identified with the root system $R_{t}$ on $X_{t}$ (cf. Definition \ref{Def_root}).

\begin{definition}\normalfont
    We call $\mathcal{T}_v$ the locus of primitive vanishing cycles. 
\end{definition}

Note the definition of $\mathcal{T}_v$ above agrees with the Definition \ref{Def_Tv} in the general situation because of the following proposition. We will introduce the general definition in the Appendix. In particular, for a hyperplane section of a cubic threefold, a root is a primitive vanishing cycle. We will use the two terms interchangeably.
   
\begin{proposition}
    The covering space \eqref{Eqn_LocusPV_CT} is connected.
\end{proposition}
\begin{proof}

Consider the set of pairs of skew lines on smooth hyperplane sections
\begin{equation}\label{eqn_M=L1L2t}
    (F\times F)^{\circ}=\{(L_1,L_2,t)\in F\times F\times U|L_1, L_2\subseteq Y_t, L_1\cap L_2=\emptyset\}.
\end{equation}

The projection $\pi:(F\times F)^{\circ}\to U$ to the third coordinate is a natural covering map, whose fiber over $t$ consists of pairs of skew lines on $X_t$. Then Proposition \ref{VC=[L]-[M]Prop} implies there is a 6-to-1 covering map $e:(L_1,L_2)\mapsto [L_1]-[L_2]$ over $U$.
\begin{figure}[ht]
    \centering
    \begin{equation}\label{Diagram_MRU}
\begin{tikzcd}
(F\times F)^{\circ}\arrow[dr,"\pi"] \arrow[r,"e"] &  \mathcal{T}_v \arrow[d,"\pi_v"]\\
&U.
\end{tikzcd}
    \end{equation}
\end{figure}

Now it suffices to show $(F\times F)^{\circ}$ is connected. Since any pair of disjoint lines $(L_1,L_2)$ spans a hyperplane in $\mathbb P^4$ and determines the hyperplane section $X_t$ containing both of the lines, the projection of \eqref{eqn_M=L1L2t} to the first two coordinate $(F\times F)^{\circ}\hookrightarrow F\times F$ is an inclusion. Consequently, $(F\times F)^{\circ}$ is a complement of a divisor, and therefore connected. 
\end{proof}

Alternatively, the connectivity also follows from that monodromy group permuting the 27 lines of the universal family of smooth hyperplane sections $\mathcal{X}^{\textup{sm}}\to U$ of $X$ is isomorphic to the Weyl group $W(\mathbb E_6)$ (c.f. \cite[VI.20]{Segre}, \cite[Theorem 0.1]{Cheng}). 

\



\subsection{Eckardt Points}
\begin{definition}\normalfont
An \textit{Eckardt point} $p$ on a cubic threefold $X$ is a closed point such that infinitely many lines on $X$ pass through $p$.  
\end{definition}

\begin{proposition}\label{Prop_Eckardt}
   The following statements are equivalent:
\begin{itemize}
    \item[(1)] $p\in X$ is an Eckardt point.
    \item[(2)] Lines on $X$ through $p$ form an elliptic curve.
    \item[(3)] The tangent hyperplane section $T_pX\cap X$ of $X$ at $p$ is a cone over a smooth plane cubic curve.
    \item[(4)] The hyperplane $H$ tangents to $X$ at $p$ and $H\cap X$ has an elliptic singularity. 
    
\end{itemize}  
\end{proposition}
\begin{proof}
$(1)\leftrightarrow(2)\leftrightarrow(3)$ is due to \cite[Lemma 8.1]{CG}. $(3)\leftrightarrow(4)$ is by classification of normal cubic surface below. 
\end{proof}

\begin{lemma}\label{classicication of cubic surfaces} (cf. \cite{Bruce-Wall} and \cite[section 9.2.2]{Dolgachev})
Let $S$ be a normal cubic surface, then

(i) $S$ has at worse ADE singularities and has at most 27 lines, or

(ii) $S$ has an elliptic singularity and has a one-parameter family of lines.
\end{lemma}

Consequently, an Eckardt point $p\in X$ corresponds to an elliptic curve $E\subseteq F$. We will unspokenly use this correspondence throughout the paper. Conversely, according to \cite{Rou}, any elliptic curve of $F$ arises from such a way. There are at most finitely many Eckardt points on a smooth cubic threefold (the maximal number is 30 reached by Fermat cubic) \cite[p.315]{CG}. 

The following lemma is well-known. I learned the proof from Dennis Tseng.

\begin{lemma}\label{Eckardt Lemma}
A general cubic threefold $X$ has no Eckardt points.
\end{lemma}

\begin{proof}
Denote by $C$ the locus in the universal family $\mathbb P^{19}=\mathbb P(\textup{Sym}^3\mathbb C^4)$ of cubic surfaces that parameterizes cone over plane cubic curves. Then $\dim C=12$. Let $W$ be the space of all cubic surfaces in $\mathbb P^4$. Since every cubic surface sits in exactly one hyperplane section, then there is a natural projection $p: W\to (\mathbb P^4)^*$, whose fiber is isomorphic to $\mathbb P^{19}$. Set $\mathbb P^{34}=\mathbb P(\textup{Sym}^3\mathbb C^5)$ to be the space of all cubic hypersurfaces in $\mathbb P^4$. Then the map 
$$f: \mathbb P^{34}\times (\mathbb P^4)^*\to W,\ (X,H)\mapsto X\cap H,$$
by sending a cubic threefold to a hyperplane section has a constant fiber dimension 15.  Let $\mathcal{C}\subseteq W$ be the locus of the cone over plane cubic curves, then $\textup{codim}_{W}\mathcal{C}=7$. Therefore the preimage $f^{-1}(\mathcal{C})$ has codimension $7$ as well. It follows that its image in $\mathbb P^{34}$ under the projection to the first coordinate has codimension at least 3, which completes the proof.
\end{proof}

\subsection{Abel-Jacobi Map}
The intermediate Jacobian $J(X)$ of a smooth cubic threefold $X$ is a principally polarized abelian variety of dimension 5. It has a theta divisor $\Theta$ unique up to a translation. Beauville \cite{Beauville} showed that $\Theta$ has a unique singularity $0$, and the projective tangent cone $\mathbb PT_0\Theta$ is isomorphic to the cubic threefold $X$ itself. This provides an alternative proof of the Torelli theorem for cubic threefold.

\begin{lemma}\label{Lemma_BlTheta}
The blow-up $\Bl_0(\Theta)$ is smooth, with the exceptional divisor isomorphic to the cubic threefold $X$.
\end{lemma}


Denote $\Bl_{\Delta_F}(F\times F)$ the blow-up of the diagonal, then according to Beauville \cite{Beauville}, the Abel-Jacobi map extends to a morphism on the blow-ups

\begin{figure}[ht]
    \centering
    \begin{equation}\label{diagram_blowupAJ}
\begin{tikzcd}
\textup{Bl}_{\Delta_F}(F\times F)\arrow[d,"\sigma"] \arrow[r,"\tilde{\Psi}"] &  \textup{Bl}_0(\Theta) \arrow[d]\\
F\times F\arrow[r,"\Psi"]&\Theta.
\end{tikzcd}
    \end{equation}
\end{figure}





We define $\Phi$ to be the rational map
\begin{equation}\label{Eqn_FFGauss}
    \Phi:F\times F \dashrightarrow (\mathbb P^4)^*,\ (L_1,L_2)\mapsto \textup{Span}(L_1,L_2).
\end{equation}

According to \cite[13.6]{CG}, it factors through the \textit{Gauss map}
\begin{equation}\label{Eqn_GaussMap}
   \mathcal{G}:\Theta\dashrightarrow (\mathbb P^4)^*,\: p\mapsto T_p\Theta 
\end{equation}


\begin{lemma}\label{Lemma_Compactification_Tv_Fiber}
The factorization extends to a commutative diagram of morphisms
\begin{figure}[ht]
    \centering
\begin{equation}\label{AJblowupDiagram}
\begin{tikzcd}
\Bl_{\Delta_F}(F\times F) \arrow[r,"\tilde{\Psi}"]\arrow[dr,"\tilde{\Phi}"] &  \textup{Bl}_0(\Theta) \arrow[d,"\tilde{\mathcal{G}}"]\\
& (\mathbb P^4)^*.
\end{tikzcd}
\end{equation}
\end{figure}

 $\tilde{\Psi}$ extends the Abel-Jacobi map \eqref{diagram_blowupAJ}. $\tilde{\mathcal{G}}$ extends the Gauss map, whose restriction to the exceptional divisor is the dual map $X\to (\mathbb P^4)^*$, $x\mapsto T_xX$.
\end{lemma}

\begin{proposition}\label{Prop_Phitilde-fiber}
   $\tilde{\Phi}$ is generically finite. $\tilde{\Phi}^{-1}(H)$ has positive dimension if and only if $H=T_pX$ is a tangent hyperplane at an Eckardt point $p$, and $\tilde{\Phi}^{-1}(H)$ is isomorphic to $E\times E$, where $E$ is the elliptic curve associated to $p$.
\end{proposition}

This can be deduced from an argument using \cite[Lemma 12.16]{CG}, but here we would like to provide an interpretation using the Hilbert scheme.

\subsection{Hilbert Scheme of a Pair of Skew Lines} \label{Sec_HS}The product $F\times F$ may be regarded as the \textit{Chow variety} of ordered pairs of lines on $X$: A general point parameterizes two skew lines, a codimension one point parameterizes pairs of incidental lines, and a codimension two set (diagonal) parameterizes double lines. However, the family parameterized by $F\times F$ is not flat: The Hilbert polynomial of $L_1\cup L_2$ is $2n+2$ when the two lines are distinct, and $2n+1$ when they intersect at a point. When $L_1=L_2$, it becomes a double line, and for different double structures, the Hilbert polynomials can be different.

To obtain a flat family, one has to resort to the notion of the Hilbert scheme.

\begin{definition}\normalfont 
    Let $H(X)$ be the irreducible component of the Hilbert scheme of $X$ containing a pair of skew lines. Call $H(X)$ the \textit{Hilbert scheme of a pair of skew lines} of $X$.
\end{definition}

\begin{theorem} \cite[Theorem 4.1]{YZ_SkewLines} \label{Thm_SkewLines}
    $H(X)$ is smooth and isomorphic to blow-up $\Bl_{\Delta_F}\Sym^2F$ of symmetric square of $F$ on the diagonal.
\end{theorem}
 Consequently, $\Bl_{\Delta_F}(F\times F)$, as a branched double cover of $H(X)$, parameterizes pairs of skew lines as well as their flat degenerations with an order. So the blow-up map $\sigma$ is the Hilbert-Chow morphism up to a double cover.

$H(X)$  parameterizes four types of subschemes with constant Hilbert polynomial $2n+2$:

\begin{itemize}
    \item[(I)] A pair of skew lines;
    \item[(II)] A line with a double structure remembering the normal direction to a quadric surface;
    \item[(III)] A pair of incident lines with an embedded point at the intersection;
    \item[(IV)] A line with a double structure remembering the normal direction to a plane, together with an embedded point on the line.
\end{itemize}

\begin{figure}[h]
     \centering 
     \begin{subfigure}[b]{0.2\textwidth}
         \centering
         \includegraphics[width=\textwidth]{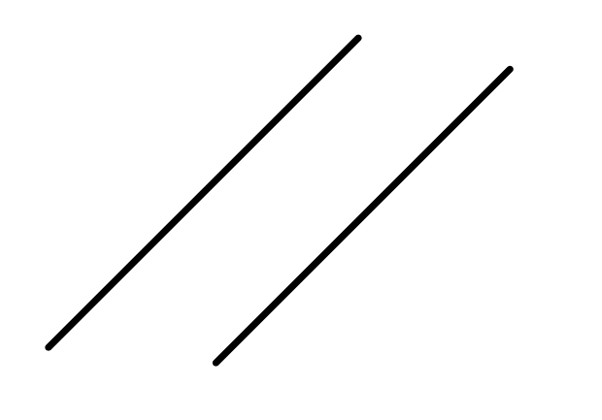}
         \caption*{Type (I)}
         \label{fig:Type (I)}
     \end{subfigure}
     \hspace{6em}
     \begin{subfigure}[b]{0.17\textwidth}
         \centering
         \includegraphics[width=\textwidth]{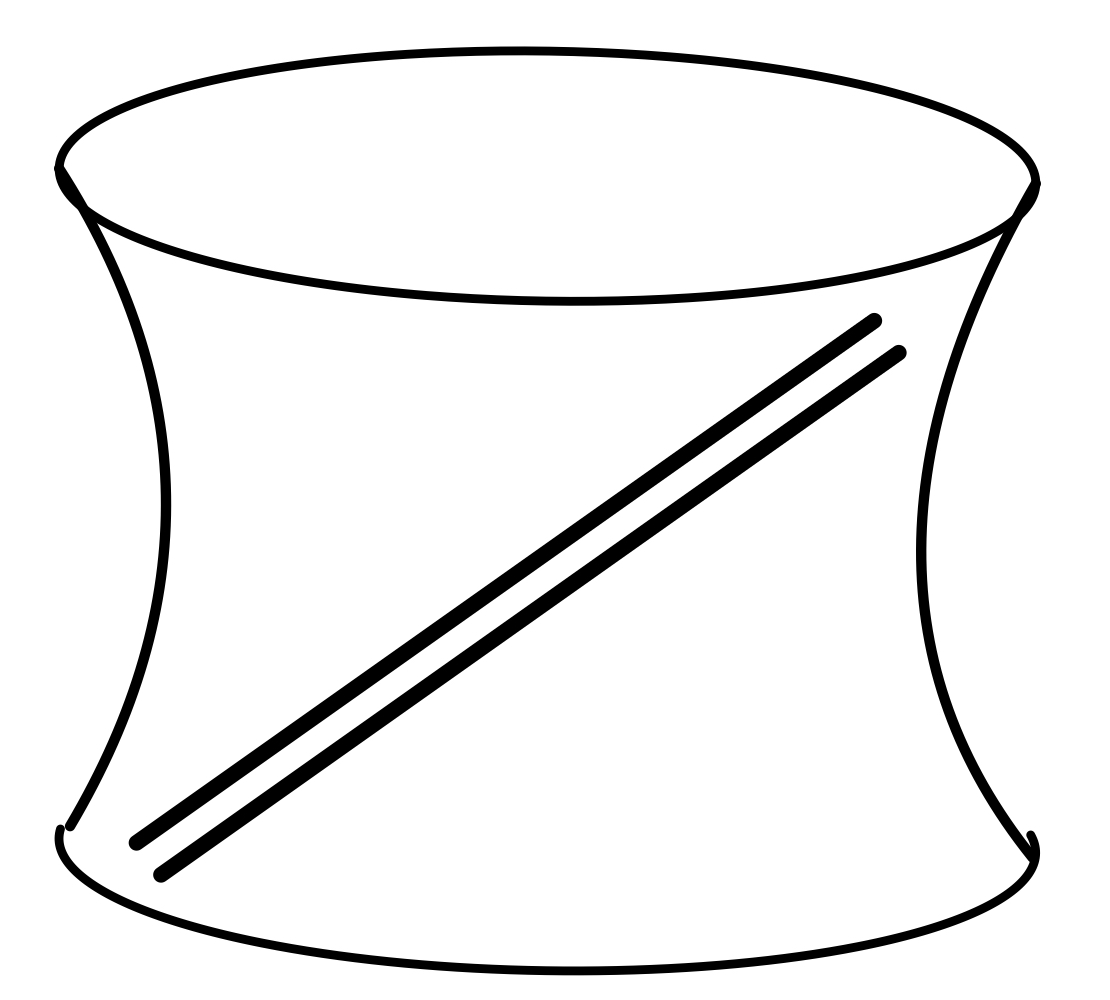}
         \caption*{Type (II)}
         \label{fig:Type (II)}
     \end{subfigure}\\
   \begin{subfigure}[b]{0.2\textwidth}
         \centering
         \includegraphics[width=\textwidth]{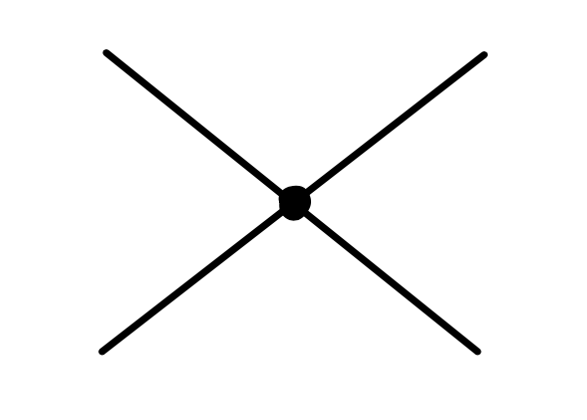}
         \caption*{Type (III)}
         \label{fig:Type (III)}
     \end{subfigure}
     \hspace{6em}
     \begin{subfigure}[b]{0.2\textwidth}
         \centering
         \includegraphics[width=\textwidth]{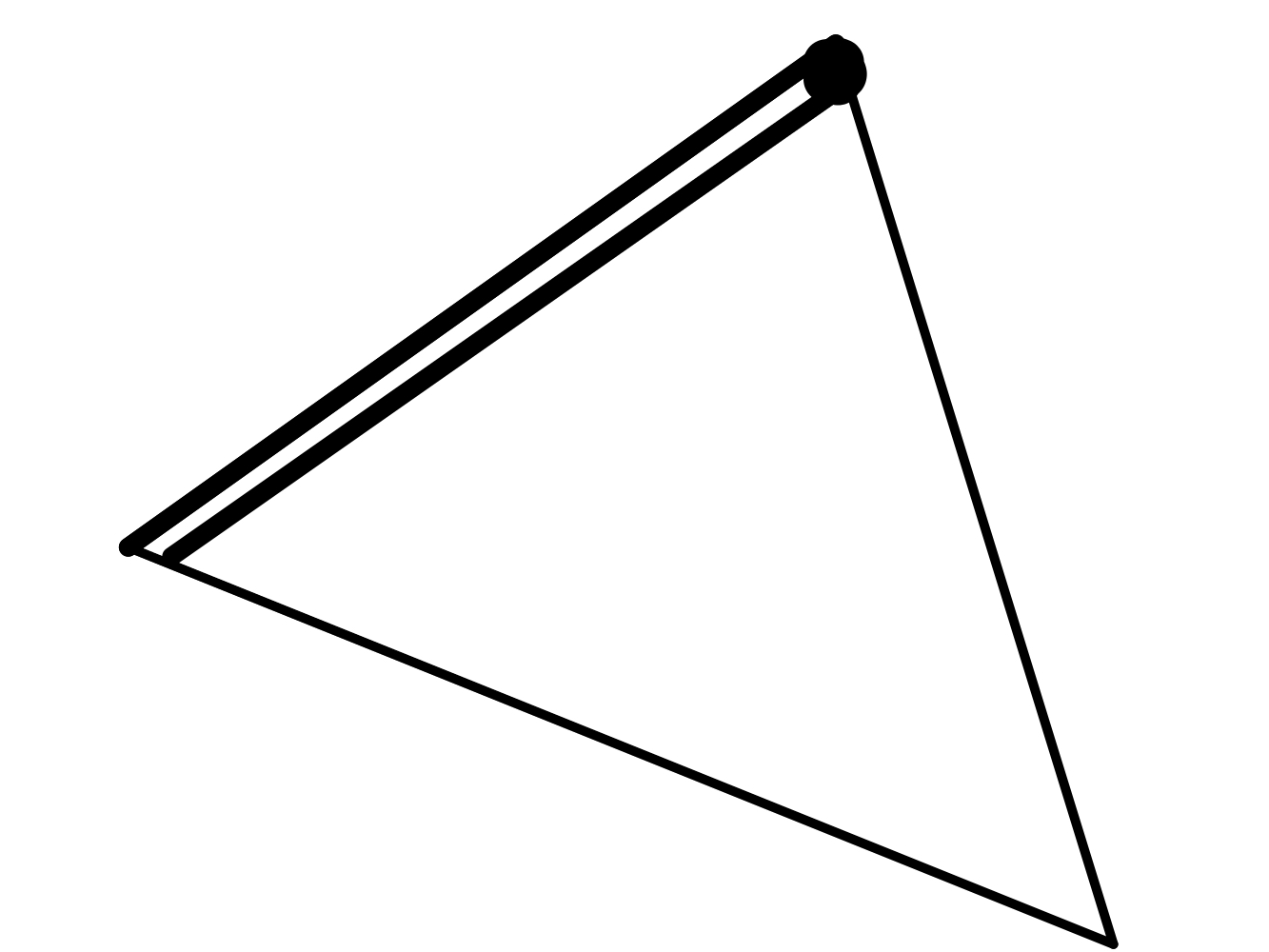}
         \caption*{Type (IV)}
         \label{fig:Type (IV)}
     \end{subfigure}
        \caption{Schemes of the Four Types}
        \label{fig:Type}
\end{figure}

 As a consequence of Theorem \ref{Thm_SkewLines}, the Hilbert-Chow morphism $\sigma$ is isomorphic on the type (I) and (III) locus, and is $\mathbb P^1$-to-1 on the locus of type (II) and (IV) schemes.

In fact, when $L_1\cap L_2=\{p\}$, there is a unique type (III) subscheme of $X$ supported on $L_1\cup L_2$: the embedded point is contained in the tangent hyperplane $T_pX$. The type (II)/(IV) subscheme, on the other hand, puts different double structures on a single line and determines the normal bundle $N_{L|X}$: 
$$N_{L|X}\cong 
\begin{cases}
    \mathcal{O}\oplus \mathcal{O},\ L\ \textup{supports a type (II) subscheme};\\
    \mathcal{O}(-1)\oplus \mathcal{O}(1),\ L\ \textup{supports a type (IV) subscheme}.
\end{cases}
$$

Just as a pair of skew lines span a hyperplane, each scheme of type (II)-(IV) is contained in a unique hyperplane. So these nonreduced schemes as flat limits of type (I) schemes, generalize the notion of "a pair of skew lines". This provides a modular interpretation of the morphism (cf. Lemma \ref{Lemma_Compactification_Tv_Fiber})
\begin{proposition} \label{Prop_HS-to-P^4}\cite[Corollary 4.2]{YZ_SkewLines}
There is a morphism
    $$\tilde{\Phi}:\Bl_{\Delta_F}(F\times F)\to (\mathbb P^4)^*$$
    that sends each subscheme $Z\in H(X)$ with an order to the unique hyperplane that contains $Z$.
\end{proposition}

In particular,  it provides an alternative proof of Proposition \ref{Prop_Phitilde-fiber}: 

\noindent\textit{Proof of Proposition \ref{Prop_Phitilde-fiber}.}
The fiber $\tilde{\Phi}^{-1}(H)$ is identified with schemes $Z\in H(X)$ contained in the cubic surface $X_H=X\cap H$. Since each pair of distinct lines supports at most one subscheme of $X_H$ of type (I) or (III), and a double line supports at most two subscheme of $X_H$ of type (II) or (IV), the fiber $\tilde{\Phi}^{-1}(H)$ is finite if $X_H$ has finitely many lines. So by the classification theorem of normal cubic surfaces (cf. Lemma \ref{classicication of cubic surfaces}), $\tilde{\Phi}^{-1}(H)$ is positive dimensional only when $H\cap X$ is a cone over an elliptic curve. In this case,  $X_H$ is a cone over elliptic curve $E$ and each pair of lines supports a unique type (III)/(IV) scheme with the embedded points supported on the cone point. So $\tilde{\Phi}^{-1}(H)$ is isomorphic to $E\times E$.\qed\\


\begin{lemma}(\cite[Proposition 4.3]{YZ_SkewLines})\label{lemma_Psi_HS}
    The restriction of $\tilde{\Psi}$ to the exceptional divisor 
    $\mathcal{E}$ of \eqref{diagram_blowupAJ} has the following modular interpretation: The map
    $$\tilde{\Psi}|_{\mathcal{E}}:\mathcal{E}\to X$$
    sends a type (II) scheme $Z\in H(X)$ to the unique point $x$ on the support line of $Z$ such that the spanning hyperplane $\tilde{\Phi}(Z)$ of $Z$ is tangent to $X$ at $x$, while it sends a type (IV) scheme $Z_p$ whose embedded point supported at $p$ to the point $\bar{p}$ on the support line such that $T_pX=T_{\bar{p}}X$.
\end{lemma}

Since $\tilde{\Psi}:\Bl_{\Delta_F}(F\times F)\to \Bl_0(\Theta)$ commutes with restriction to hyperplane sections, according to the commutativity of the diagram \eqref{AJblowupDiagram}, we also have the following interpretation of $\Bl_0(\Theta)$. (One also compares with \cite[Proposition 6.2]{YZ_SkewLines}.)

\begin{proposition}\label{Prop_Homological}
    $\Bl_0(\Theta)$ parameterizes equivalent classes of subscheme of $X$ of type (I)-(IV) with an order. 
    \begin{itemize}
        \item[(i)] For type (I) and type (III) schemes with an order, the relation is homological, i.e., the difference $[L_1]-[L_2]$ of two lines is equivalent to another iff they represent the same class in the corresponding (singular) cubic surface.
        \item[(ii)] For type (II) and type (IV) schemes, the equivalence relation is incidental: such a scheme uniquely determines a point on their support line, two schemes $Z_1\sim Z_2$ if and only if the corresponding incidental points are the same (cf. Lemma \ref{lemma_Psi_HS}).
    \end{itemize}   
\end{proposition}
\begin{proof}
   It suffices to show (i). By taking the hyperplane section at $X_H=X\cap H$, we have a cubic surface with either ADE singularities or an elliptic singularity (cf. Lemma \ref{classicication of cubic surfaces}). Let $\tilde{X}_H\to X_H$ be the minimal resolution, and denote $D$ as the exceptional divisor. There is an exact sequence of cohomology
    $$H^1(\tilde{X}_H)\xrightarrow{\alpha} H^1(D)\to H^2(X_H)\to H^2(\tilde{X}_H)\to H^2(D).$$
    
When $X_H$ has ADE singularities, $D$ is a disjoint union of bunches of rational curves, so $H^1(D)=0$. In particular, $H^2(X_H)$ consists of classes of $H^2(\tilde{X}_H)\cong \Z^7$ by setting all the $(-2)$ curves on the exceptional fibers to be zero. When $X_H$ is a cone over an elliptic curve $E$, $\alpha$ is an isomorphism, so $H^2(X_H)$ consists of classes of $H^2(\tilde{X}_H)\cong \Z^2$ by setting the class of exceptional curve $D\cong E$ to be zero. In both cases, the support lines $L_1\cup L_2$ of a type (I) or (III) subscheme of $X_H$ is the limit of pairs of skew lines $L_{1,t}, L_{2,t}$ on a one-parameter family of cubic surfaces $X_{t}$. By base change and desingularization, the lines $L_{1,t}, L_{2,t}$ specialize to (-1) curves $\tilde{L}_1$, $\tilde{L}_2$ on $\tilde{X}_H$, whose image in $X_H$ is $L_1$ and $L_2$. Therefore $[L_1]-[L_2]\sim [L_1']-[L_2']$ on $X_H$ iff $[\tilde{L}_1]-[\tilde{L}_2]\sim [\tilde{L}_1']-[\tilde{L}_2']$ on $\tilde{X}_H$ modulo classes generated by exceptional curves.

\end{proof}

\section{Compactification of the Primitive Vanishing Cycle Component}\label{Section_Tvbar}

In this section, our main goal is to understand Stein's compactification $\bar{\mathcal{T}}_v$ and prove the Theorem \ref{barT'BlowupThm}. We will also study other natural compactifications of $\mathcal{T}_v$ as well as their relationship to various compactifications of $(F\times F)^{\circ}$.

\subsection{Stein's Compactification}\label{Sec_3.Compactification}


As a consequence of a lemma (cf. Lemma \ref{analytic_cover_closure}) of Stein, given an analytic finite cover $W\to U$ and a compactification $U\subseteq \bar{U}$ on the base, there is a canonical compactification of the total space $\bar{W}$ and the resulting cover $\bar{W}\to \bar{U}$ is branched. 

Apply the lemma to our 72-to-1 cover \eqref{Eqn_LocusPV_CT} with respect to the compactification $U\subseteq (\mathbb P^4)^*$ on the base, we have

\begin{proposition}\label{Prop_barT'Exist}
There exists a normal algebraic variety $\bar{\mathcal{T}}_v$ together with a finite map 
\begin{equation}\label{Eqn_bar-piv}
\bar{\pi}_v:\bar{\mathcal{T}}_v\to (\mathbb P^4)^*
\end{equation}
which extends the 72-to-1 covering map \eqref{Eqn_LocusPV_CT} as branched analytic covers. Moreover, $\bar{\mathcal{T}}_v$ is unique up to isomorphism.
\end{proposition}

As a consequence of Proposition \ref{Prop_Phitilde-fiber}, we have
\begin{lemma}\label{lemma_ExtendGauss}
The extended Gauss map $\tilde{\mathcal{G}}:\textup{Bl}_0(\Theta)\to (\mathbb P^4)^*$ is generically finite. $\tilde{\mathcal{G}}^{-1}(t)$ has positive dimension if and only if
\begin{itemize}
    \item[(1)] $H_t\cap X$ is a cone over an elliptic curve $E$, and
    \item[(2)] $\tilde{\mathcal{G}}^{-1}(t)$ is isomorphic to $E$.
\end{itemize}
\end{lemma}
\begin{proof}
We identify $E\times E\subseteq F\times F$ with its strict transform in the blow-up. It suffices to show that the image of positive dimensional fiber $\tilde{\Phi}^{-1}(t)\cong E\times E$ under the map $\tilde{\Psi}$ is isomorphic to $E$. 

The restriction of the Abel-Jacobi map \eqref{Intro_eqn_AJ} to $E\times E$ factors through
$$\Psi|_{E\times E}: E\times E\xrightarrow{f} E\xrightarrow{g} Alb(F)\cong J(X),$$
where $f$ is given by $(p,q)\mapsto p-q$. The morphism $g$ is a closed immersion since it factors through $F\to Alb(F)$, which is a closed immersion \cite[Theorem 4]{Beauville2}. The last isomorphism is due to \cite[Theorem 11.19]{CG}.

It follows that $\Psi(E\times E)\cong E$ is an elliptic curve, and by commutativity of \eqref{AJblowupDiagram}, the strict transform of $\Psi(E\times E)$ is contracted by $\tilde{\mathcal{G}}$.
\end{proof}

\begin{remark}\normalfont

   The inclusion of an elliptic curve $E\subseteq J(X)$ as above makes the intermediate Jacobian $J(X)$ reducible. The principal polarization on $J(X)$ restricts to twice the principal polarization on $E$ \cite[Lemma 1.12]{CMZ} \cite[Proposition 4.8.1]{YZ_thesis}.
    
\end{remark}


\begin{theorem}\label{barT'BlowupThm}
The extended Gauss map $\tilde{\mathcal{G}}:\Bl_{0}(\Theta)\to (\mathbb P^4)^*$ factors through a birational morphism $\textup{Bl}_0(\Theta)\to \bar{\mathcal{T}}_v$, which contracts finitely many elliptic curves corresponding to the Eckardt points on the cubic threefold $X$. Moreover, the fiber of $\bar{\mathcal{T}}_v\to (\mathbb P^4)^*$ over an Eckardt hyperplane is a single point.
\end{theorem}

\begin{proof}

The Abel-Jacobi map $\Psi:F\times F\to J(X)$ is generically 6-to-1, and the image is the theta divisor $\Theta$. If we restrict $\Psi$ to the open subspace $(F\times F)^{\circ}$ parameterizing pairs of skew lines, then $\Psi|_{(F\times F)^{\circ}}$ factors through $\mathcal{T}_v$, because the six differences of two disjoint lines are in the same rational equivalent classes, or the nature of the Topological Abel-Jacobi map \cite[Definition 2.1.2]{Zhao}. Further, we note that these maps preserve projection to $U$, so we have the following diagram.

\begin{figure}[ht]
    \centering
\begin{tikzcd}
(F\times F)^{\circ}/U\arrow[r,"e"] \arrow[dr,"\Psi|_{(F\times F)^{\circ}}"'] & \mathcal{T}_v/U \arrow[d,"j"] \\
& \Theta^{\circ}/U
\end{tikzcd}
\end{figure}{}

Here $\Theta^{\circ}$ is the image of $(F\times F)^{\circ}$, which is dense open in $\Theta$. The evaluation map $e$ sends $(L_1,L_2)\mapsto [L_1]-[L_2]$ to the difference of the classes. Since $e$ is also 6-to-1, $j$ is an isomorphism. It follows that $\mathcal{T}_v\cong \Theta^{\circ}\hookrightarrow \textup{Bl}_0(\Theta)$ is an open dense subspace. 

Now we take the Stein factorization
$$\textup{Bl}_0(\Theta)\xrightarrow{\tilde{\mathcal{G}}_1} W\xrightarrow{\tilde{\mathcal{G}}_2} (\mathbb P^4)^*,$$
where $\tilde{\mathcal{G}}_1$ is birational and contracts the curves $E\cong\tilde{\Psi}(E\times E)$, while $\tilde{\mathcal{G}}_2$ is finite. Moreover, $W$ is normal since $\textup{Bl}_0(\Theta)$ is normal (cf. \cite[p.213]{GR84}).

Now, the composition $\mathcal{T}_v\hookrightarrow \textup{Bl}_{0}\Theta\to W$ is inclusion and preserves the projection to $(\mathbb P^4)^*$. So as branched covering maps, $W/(\mathbb P^4)^*$ extends $\mathcal{T}_v/U$. Due to the normality of $W$ and the uniqueness of extension analytic branched covering from Lemma \ref{analytic_cover_closure}, $W/(\mathbb P^4)^*\cong \bar{\mathcal{T}}_v/(\mathbb P^4)^*$. This proves the theorem.
\end{proof}

For the rest of the section, we will discuss the relationship of $\bar{\mathcal{T}}_v$ to other compactifications.

\subsection{Nash Blow-up}
\begin{definition}\normalfont
The \textit{Nash blow-up} of $\Theta$ is the graph closure $\hat{\Theta}$ of the Gauss map $\mathcal{G}:\Theta\dashrightarrow (\mathbb P^4)^*$ \eqref{Eqn_GaussMap}. Similarly, we call the graph closure $\widehat{F\times F}$ of the $\Phi: F\times F\dashrightarrow (\mathbb P^4)^*$ \eqref{Eqn_FFGauss} the Nash blow-up of $F\times F$.
\end{definition}

Both spaces are "minimal" modifications of $\Theta$ and $F\times F$ to assign hyperplanes continuously on the closure and extend rational maps $\mathcal{G}$ and $\Phi$ to morphisms. However, both spaces are singular, and their normalizations are spaces that we are familiar with.

\begin{proposition}
    The normalization of $\widehat{F\times F}$ is isomorphic to $\textup{Bl}_{\Delta_F}(F\times F)$. The normalization of $\hat{\Theta}$ is isomorphic to $\Bl_0(\Theta)$. 
\end{proposition}
\begin{proof}
   Let $\Gamma(\tilde{\Phi})$ be the graph of $\tilde{\Phi}: \textup{Bl}_{\Delta_F}(F\times F)\to (\mathbb P^4)^*$. Let $p_1$ be the composition of projection to the first coordinate and $\sigma$, and $p_2$ be the projection to the second coordinate. Then there is a map
   $$p_1\times p_2: \textup{Bl}_{\Delta_F}(F\times F)\cong \Gamma(\tilde{\Phi})\to F\times F\times (\mathbb P^4)^*.$$
   
   Its image is closed and irreducible and contains the graph of $\Phi$ as a Zariski dense subspace, so the image has to be $\widehat{F\times F}$. In particular, there is a generically one-to-one map
   \begin{equation}
      n_1: \textup{Bl}_{\Delta_F}(F\times F)\to \widehat{F\times F}.
   \end{equation}
   
In fact, by the Hilbert scheme interpretation of $\textup{Bl}_{\Delta_F}(F\times F)$ (cf. Theorem \ref{Thm_SkewLines}). The map $n_1$ is an isomorphism off the diagonal, one-to-one on the fiber $\sigma^{-1}(L)$ of the line $L$ with $N_{L|X}\cong \mathcal{O}\oplus \mathcal{O}$, and two-to-one on the fiber $\sigma^{-1}(L)$ of the line $L$ with $N_{L|X}\cong \mathcal{O}(1)\oplus \mathcal{O}(-1)$. Geometrically, there is a unique type (II) subscheme supported on a line with normal bundle $\mathcal{O}\oplus \mathcal{O}$, whereas there are two subschemes $Z_p$ and $Z_{\bar{p}}$ of type (IV) supported on a line with normal bundle $\mathcal{O}(1)\oplus \mathcal{O}(-1)$, where $Z_p$ and $Z_{\bar{p}}$ have embedded points supported at $p$ and $\bar{p}$, such that $T_pX=T_{\bar{p}}X$.

In particular, $n_1$ is finite and birational. So by Zariski's main theorem, it is normalization.

Similar holds for $\Bl_{0}(\Theta)\to \hat{\Theta}$, whose restriction to the exceptional divisor is identified to the dual map $X\to X^*$, $x\mapsto T_xX$. 
\end{proof}

\begin{corollary} The map $\tilde{\Phi}$ and the extended Gauss map $\tilde{\mathcal{G}}$ \eqref{AJblowupDiagram} factor through the Nash blow-ups, and there is a commutative diagram 

\begin{figure}[ht]
    \centering
\begin{equation} \label{diagram_NashBlowup}
\begin{tikzcd}
\textup{Bl}_{\Delta_F}(F\times F)  \arrow[r,"\tilde{\Psi}"] \arrow[d,"n_1"] & \Bl_0(\Theta) \arrow[d,"n_2"]\arrow[dr,"\tilde{\mathcal{G}}"]\\
\widehat{F\times F}\arrow[r]& \hat{\Theta}\arrow[r] &(\mathbb P^4)^*,
\end{tikzcd}
\end{equation}
\end{figure}{}
\noindent where $n_1$ and $n_2$ are normalizations.
\end{corollary}

\subsection{Relations Among Compactifications of $(F\times F)^{\circ}$ and $\mathcal{T}_v$}
So far, we found four compactifications of $(F\times F)^{\circ}$:
\begin{itemize}
    \item[(a)] The Chow variety $F\times F$ parameterizing order pairs of lines on $X$,

\item[(b)]  the nested Hilbert scheme $\Bl_{\Delta_F}(F\times F)$ of a pair of skew lines on $X$,

\item[(c)] the Stein completion $\overline{(F\times F)^{\circ}}$, obtained by applying Stein's completion to $(F\times F)^{\circ}\to U$, and 
\item[(d)] the Nash blow-up $\widehat{F\times F}$.

\end{itemize}

The four compactifications are related by the following diagram:

\begin{figure}[ht]
    \centering
    \begin{equation}\label{diagram_M-Completion}
\begin{tikzcd}
\Bl_{\Delta_F}(F\times F) \arrow[r,"c"]\arrow[d,"n_1"]\arrow[dr,"\sigma"]&\overline{(F\times F)^{\circ}} \arrow[d,"f"]\\
\widehat{F\times F}\arrow[r]&F\times F.
\end{tikzcd}
    \end{equation}
\end{figure}

$\sigma$ is the blow-up along the diagonal. The Stein factorization of $\sigma$  factors through $\overline{(F\times F)^{\circ}}$, where $c$ contracts finitely many abelian surfaces of the form $E_i\times E_i$ and $f$ is finite. Again $E_i$ is an elliptic curve corresponding to an Eckardt point on $X$ and $c$ will be an isomorphism when the cubic threefold $X$ is general.

Similarly, there are four compactifications of $\mathcal{T}_v$:
\begin{itemize}
    \item[(a)] the theta divisor $\Theta$,
    \item[(b)] the blow-up of the theta divisor $\Bl_0(\Theta)$,
    \item[(c)] the Stein completion $\bar{\mathcal{T}}_v$, and 
    \item[(d)] the Nash blow-up $\hat{\Theta}$.
\end{itemize}

They are related by the right face of the following three-dimensional diagram, which also includes the diagram \eqref{diagram_NashBlowup} and \eqref{diagram_M-Completion} the back face and the left face.

\begin{equation}\label{diagram_total}
\begin{tikzcd}[row sep=1.5em, column sep = 1.5em]
\Bl_{\Delta_F}(F\times F)\arrow[rr,"\tilde{\Psi}"] \arrow[dr,"c_1"] \arrow[dd,swap,"n_1"] &&
\Bl_0(\Theta) \arrow[dd,swap,near start,"n_2"] \arrow[dr,"c_2"] \\
& \overline{(F\times F)^{\circ}}\arrow[rr, crossing over] &&
\bar{\mathcal{T}}_v  \arrow[d] \\
\widehat{F\times F} \arrow[rr] \arrow[dr] && \hat{\Theta} \arrow[dr]\arrow[r]&(\mathbb P^4)^*   \\
& F\times F \arrow[rr,"\Psi"] && \Theta\arrow[u,dashed,swap,"\mathcal{G}"] 
\arrow[from=2-2, to=4-2,crossing over]
\end{tikzcd}
\end{equation}

$c_1$ and $c_2$ on the top face are contractions and arise from Stein factorization. Also, $\overline{(F\times F)^{\circ}}\to (\mathbb P^4)^*$ factors through $\bar{\mathcal{T}}_v$ due to uniqueness of Stein completion.

\begin{proposition} The diagram 
\eqref{diagram_total} is commutative. The horizontal arrows are induced by Abel-Jacobi map \eqref{Intro_eqn_AJ} and are generically finite. All vertical regular maps are finite, and all diagonal maps are birational.

\end{proposition}

\section{Boundary Points and Minimal Resolutions} \label{Section_Cubic3fold_Boundary}

In this section, we would like to understand the geometric meaning of the fiber $\bar{\mathcal{T}}_v\to (\mathbb P^4)^*$ over a point $t_0$ where hypersection $X_{t_0}$ is singular (cf. Definition \ref{Intro_Def_LPVS}). For example, what is the relation of these boundary points to the singularities of $X_{t_0}$ and how many of them are they?

Our starting point is the following observation: Let $\{S_t\}$ be a one-parameter family of surfaces degenerating to a surface with an ordinary node when $t=0$, then by a base change and desingularization, the vanishing cycle specializes to an effective $(-2)$-curve (cf. Example \ref{Exa_CubicNode}). In the spirit of Brieskorn's resolution, there is a similar picture for the degeneration of surfaces with ADE singularities. Based on this observation, and the monodromy of Milnor fiber of hypersurfaces with ADE singularities by Arnold et al., we will describe the boundary points of the compactification of $\mathcal{T}_v$.

First of all, we have a topological interpretation in terms of local monodromy: For each $t_0\in (\mathbb P^4)^*$, pick a suitably small open neighborhood $B$ of $t_0$. Fix another base point $t'\in B^{\textup{sm}}:=U\cap B$, then the monodromy action on the root system $R_{X_{t'}}\cong R(\mathbb E_6)$ over $t'$ via monodromy action around $t_0$ is a homomorphism
\begin{equation}\label{eqn_LocMonodromyAction-General}
   \rho_{t_0}: \pi_1(B^{\textup{sm}},t')\to \Aut \Hv^2(X_{t'},\Z)\cong W(\mathbb E_6),
\end{equation}
which is the same as the automorphism of the root system.

\begin{definition}\normalfont\label{Def_localMonGrp_compact}
We call $G_{t_0}=\textup{Im}(\rho_{t_0})$ the \textit{(local) monodromy group} on the root system of the family of hyperplane sections of $X$ around $X_{t_0}$. 
\end{definition}

By normality of completion of finite analytic cover in Lemma \ref{analytic_cover_closure}, our first observation is
\begin{lemma}\label{Lemma_Cubic3fold_LocMonodromyOrbit}
There is a bijection between the fiber of $\bar{\mathcal{T}}_v\to (\mathbb P^4)^*$ at $t_0$ and the set of orbits of monodromy action $G_{t_0}\acts R(\mathbb E_6)$.
\end{lemma}

On the other hand, the hyperplane section $X_{t_0}$ is a normal cubic surface and has either ADE singularities or an elliptic singularity (cf. Lemma \ref{classicication of cubic surfaces}).  So to provide some geometric understanding of the compactification $\bar{\mathcal{T}}_v$ as proposed in Question \ref{Question_main}, we would like to answer the following question:

\begin{question}\label{Question_Boundary}
 What is the relation between the orbits of monodromy action $G_{t_0}\acts R(\mathbb E_6)$ and the singularities of the cubic surface $X_{t_0}$?
\end{question}

For an isolated hypersurface singularity of ADE type, it is known by Arnold and Gabrelov (cf. Lemma  \ref{GabrielovLem}) that the monodromy group is isomorphic to the Weyl group corresponding to the minimal resolution.

\subsection{Minimal Resolution}
Let $S$ be a cubic surface with at worst ADE singularities. Let $$\sigma:\tilde{S}\to S$$ be its minimal resolution, then $\tilde{S}$ is a weak del Pezzo surface of degree $3$ and one can still define the root system $R(\tilde{S})$ as 
\begin{equation}
    \alpha^2=-2,~\alpha\cdot K_{\tilde{S}}=0. \label{rootsystemeqn}
\end{equation}

It is known that $R(\tilde{S})$ is isomorphic to $R(\mathbb E_6)$. This can be seen as follows. One can regard $\tilde{S}$ as blowing up six bubble points on $\mathbb P^2$ in an almost general position. The six points can be deformed smoothly as they move to a general position along a real one-dimensional path. As the equation \eqref{rootsystemeqn} is a topological invariant, the root system on $\tilde{S}$ is defined.

Note that each irreducible component $C$ of the exceptional divisor of $\sigma$ is a $(-2)$ curve and is orthogonal to the class $K_{\tilde{S}}$ since $\sigma$ is crepant, so $C$ defines a root and is effective as divisor class. We call such root an \textit{effective root}. The set of all effective roots generates a sub-root system $R_e$ of $R(\tilde{S})$. Since each of the singularity, $x_i$ of $S$ corresponds to a bunch of $(-2)$-curves on $\tilde{S}$ and they generate a sub-root system $R_i$, $R_e$ is isomorphic to the product $\prod_{i\in I}R_i$, where $I$ is the index set of singularities of $S$. Each $R_i$ corresponds to a connected sub-diagram of the Dynkin diagram of $\mathbb E_6$. One refers to \cite{Dolgachev}, sections 8.1, 8.2, 8.3, and 9.1 for the detailed discussion.

Moreover, the reflections with respect to all of the effective roots define a subgroup $W(R_e)$ of the Weyl group $W(\mathbb E_6)$, and $W(R_e)$ is isomorphic to the product $\prod_{i\in I}W_i$, where $W_i$ is the Weyl group generated by the reflections corresponding to the exceptional curves over the singularity $x_i$. One can consider the action of $W(R_e)$ on $R(\mathbb E_6)$. The orbits that are contained in $R_e$ are naturally in bijection with the set of singularities on $S$. One can also define the maximal/minimal root of an orbit. In particular, the maximal root of the orbit corresponding to $x_i$ equals the cohomology class of the fundamental cycle $Z_i$ at $x_i$. The readers can refer to \cite{LLSvS}, section 2.1 for details.

\begin{definition}\normalfont
Let $R(S)$ be the set of the orbits in $R(\mathbb E_6)$ under the action of $W(R_e)$. We call $R(S)$ to be the root system on $S$.
\end{definition}

Note that $R(S)$ is just a set, without any intersection pairing. In fact, the set $R(S)=R(\mathbb E_6)/W(R_e)$ is used in \cite[Theorem 2.1]{LLSvS} to parameterize the connected components of the (reduced) Hilbert scheme of generalized twisted cubics on $S$.

\subsection{Main Results}
Our key proposition in this section is the following.
\begin{proposition} \label{MonoActGloProp}  Assume that the cubic surface $X_{t_0}$ has at worst ADE singularities. Then there is an isomorphism
\begin{equation}
    G_{t_0}\cong W(R_e)
\end{equation}
between the monodromy group $G_{t_0}$ of hyperplane sections of $X$ near $X_{t_0}$ and the subgroup of $W(\mathbb E_6)$ generated by the reflections of all effective roots in the minimal resolution $\tilde{X}_{t_0}$ of $X_{t_0}$.
\end{proposition}

\begin{remark}\normalfont
    Note by definition, $W(R_e)$ is the product of Weyl groups of type ADE corresponding to the singularities on the cubic surface $X_{t_0}$, so each of the summands is the monodromy group of the deformation of surface singularity of type ADE over its semi-universal deformation space. Therefore, Proposition \ref{MonoActGloProp} implies that the local monodromy group of $X_{t_0}$ in the four-dimensional space as hyperplane varies is the "maximal" one. 
    
  \end{remark}

So we answer Question \ref{Question_Boundary} with the following theorem.
\begin{theorem} \label{MonOrbThm}
There is a bijection of sets
$$PV_{t_0}\longleftrightarrow R(X_{t_0}).$$
\end{theorem}
\begin{proof}
This follows from Lemma \ref{Lemma_Cubic3fold_LocMonodromyOrbit} and Proposition \ref{MonoActGloProp}.
\end{proof}
\begin{example} \normalfont
Let $\{X_t\}_{t\in \Delta}$ be a one-parameter family of cubic surface with $X_{0}$ having an $A_1$ singularity as considered in Example \ref{Exa_CubicNode}. Then there is a vanishing cycle $\delta$ supported on a nearby fiber $X_{t}$ where $t\in \Delta^*$. The monodromy representation $\pi_1(\Delta^*,t)\to \Aut R(\mathbb E_6)$ is generated by the Picard-Lefschetz transformation
\begin{equation}
   T_{\delta}:\alpha\mapsto \alpha+(\alpha,\delta)\delta. \label{PLeqn}
\end{equation}
$T_{\delta}$ has order two and coincides with the reflection of the root system along $\delta$. So the orbits of local monodromy action are identified with the orbits of Weyl group $W(A_1)=\mathbb Z_2$ action.

Use the same notation as Proposition \ref{VC=[L]-[M]Prop} and choose the vanishing cycle $\delta=2h-e_1-\cdots-e_6$. The 72 roots can be expressed as different sets of classes

(1) $\pm \delta$, $2$ roots;

(2) $\pm(h-e_i-e_j-e_k)$, $i,j,k$ distinct, $40$ roots;

(3) $e_i-e_j, i\neq j$, $30$ roots.

The roots in (1) and (2) have nonzero intersections with $\delta$, so they correspond to connected 2-to-1 covers of $\Delta^*$. The roots in (3) are orthogonal to $\delta$ and correspond to trivial covers of $\Delta^*$, so by the Picard-Lefschetz formula, the number of monodromy orbits is $(2+40)/2+30=51$. So $|PV_{0}|=|R(X_{0})|=51$.

Using Proposition \ref{Prop_Homological}, the orbit in (1) is represented by a type (II) scheme, orbits in (2) are represented by type (I) schemes, while the orbits in (3) are represented by type (III) schemes.

\end{example}

In \cite[Theorem 2.1]{LLSvS}, the authors showed that there is a one-to-one correspondence between $R(S)=R(\mathbb E_6)/W(R_e)$ are and the connected components of the reduced Hilbert schemes of generalized twisted cubics on $S$. The orbits that contain an effective root correspond to generalized twisted cubics that are not Cohen-Macaulay (whose reduced schemes are planar). The orbits without any effective roots correspond to the generalized twisted cubics that are arithmetic Cohen-Macaulay (whose reduced schemes are not planar).  \cite[Section 3]{LLSvS} showed that there is a bijective between the $W(R_e)$-orbit on $R(\tilde{S})\setminus R_e$ and the linear determinantal representations of cubic surfaces. The cardinality of such orbits is listed on \cite[p.102, Table 1]{LLSvS}. On the other hand, we know that the cardinality of the orbits on $R_e$ is exactly the number of the singularities. So we obtain the cardinality of the root system $R(S)$ by adding up the two numbers.

\begin{corollary}
Let $X_{t_0}$ be a cubic surface with ADE singularities arising from a hyperplane section of cubic threefold $X$. Then the cardinality $\#=|PV_{t_0}|=|R(X_{t_0})|$  of primitive vanishing cycles (the root system) on $X_{t_0}$ is listed in the table below.
\end{corollary}

\begin{table}[!ht]
\begin{center}
\begin{tabular}{c|c|r||c|c|r||c|c|r}
$R_e$&Type&\#&$R_e$&Type&\#&$R_e$&Type&\#\\\hline
$\emptyset$&I&72&$4A_1$&XVI&17&$A_1+2A_2$&XVII&9\\
$A_1$&II&51&$2A_1+A_2$&XIII&15&$A_1+A_4$&XIV&6\\
$2A_1$&IV&36&$A_1+A_3$&X&12&$A_5$&XI&5\\
$A_2$&III&31&$2A_2$&IX&14&$D_5$&XV&3\\
$3A_1$&VIII&25&$A_4$&VII&9&$A_1+A_5$&XIX&3\\
$A_1+A_2$&VI&22&$D_4$&XII&7&$3A_2$&XXI&5\\
$A_3$&V&17&$2A_1+A_3$&XVIII&8&$E_6$&XX&1\\
\end{tabular}\\[1ex]
\end{center}
\caption{Numbers of primitive vanishing cycles on cubic surfaces with ADE singularities.}
\end{table}

\begin{remark}\normalfont
    According to Lemma \ref{Lemma_BlTheta}, the exceptional divisor of $\Bl_0(\Theta)$ is isomorphic to $X$, and it is ismorphic onto its image under the contraction $\Bl_0(\Theta)\to \bar{\mathcal{T}}_v$ (cf. Theorem \ref{barT'BlowupThm}). The component $X\subseteq \bar{\mathcal{T}}_v$ parameterizes equivalent classes of effective roots: There is a commutative diagram
\begin{figure}[ht]
    \centering
\begin{equation} 
\begin{tikzcd}
X \arrow[r,hookrightarrow] \arrow[d,"\mathcal{D}"] & \bar{\mathcal{T}}_v\arrow[d,"\bar{\pi}_v"]\\
X^*\arrow[r,hookrightarrow]& (\mathbb P^4)^*.
\end{tikzcd}
\end{equation}
\end{figure}{}

    The restriction of $\Bar{\pi}_v$ to $X$ is isomorphic to the dual map. The fiber of $\mathcal{D}$ over $t_0$ corresponds to the vanishing cycles associated with the singularities on $X_{t_0}$. These vanishing cycles are exactly the $W(R_e)$-orbits on the effective roots. So we can say the component $X\subseteq \bar{\mathcal{T}}_v$ parameterizes effective limiting primitive vanishing cycles.
\end{remark}

\subsection{A Local Argument}

 We will start to prove Proposition \ref{MonoActGloProp} from this section. We need first to study the local monodromy of Milnor fiber of a single singularity on a cubic surface $X_{t_0}$.

Let $p:\mathcal{X}\to B$ be the family of cubic surfaces over the ball $B$ arising from hyperplane sections on $X$. Let $x_0$ be an isolated singularity of $X_{t_0}$, where $X_{t_0}$ is the hyperplane section $X\cap H_{t_0}$. Take a small ball $D_0$ in the total space $\mathcal{X}$ around $x_0$. Then by restricting to $D_0^{\textup{sm}}=D_0\setminus p^{-1}(X^*)$, the morphism
$$p^{\textup{sm}}:D_0^{\textup{sm}}\to B^{\textup{sm}}$$
is a smooth fiber bundle. Let $F$ be a fiber, then there is a monodromy representation
\begin{equation}
    \rho_{t_0,x_0}:\pi_1(B^{\textup{sm}},t')\to \textup{Aut}H^{2}(F,\mathbb Z). \label{MonActLoc}
\end{equation}
\begin{definition}\normalfont
We call the image $G_{t_0,x_0}:=\textup{Im}(\rho_{t_0,x_0})$ the (local) monodromy group on cohomology of Milnor fiber of the singularity $x_0$ on $X_{t_0}$.
\end{definition}

\begin{proposition} \label{LocMonGpProp}
Suppose the singularity $x_0\in X_{t_0}$ has type ADE. Then the local monodromy group $G_{t_0,x_0}$ around $x_0$ is isomorphic to the Weyl group $W_{x_0}$ of the Lie algebra that corresponds to the ADE type of the singularity $x_0$.
\end{proposition}

To prove Proposition \ref{LocMonGpProp}, we need to use the Milnor fiber theory. One refers to \cite{Durfee} for a more detailed survey. 

\subsection{Monodromy Group on Milnor Fiber}

Let $f(x_1,...,x_n)=0$ be a hypersurface in $\mathbb C^n$ with an isolated singularity at $0$, then the Milnor fiber $F$ of $f$ is the $\{f=w\}\cap B^n$ for a ball $B^n$ around origin of small radius and $w \in \mathbb C$ with a small magnitude. $F$ has homotopy type of a bouquet of $\mu$ spheres of dimension $n-1$, where $\mu$ is the Milnor number of the singularity, which coincides with the dimension of $\mathbb C$-vector space $\mathbb C[x_1,...,x_n]/(\frac{\partial f}{\partial x_1},...,\frac{\partial f}{\partial x_n})$.

A \textit{deformation} of $f$ is an analytic function $$g(x_1,...,x_n,w):\mathbb C^n\times \mathbb C\to \mathbb C$$ such that $g(x_1,...,x_n,0)=f(x_1,...,x_n)$, and $\tilde{f}(x_1,...,x_n)=g(x_1,...,x_n,1)$ is called a \textit{perturbation} of $f$. 
There is a perturbation $\tilde{f}$ of $f$ such that $\tilde{f}$ is a Morse function in the sense that all critical values of $\tilde{f}$ are distinct and all critical points are nondegenerate. There are exactly $\mu$ critical points ${t_1,...,t_{\mu}}$ of $\tilde{f}$ and they are contained in a $\delta$-neighborhood $D_{\delta}$ of $0$ in $\mathbb C$. The Milnor fibers of $f$ and $\tilde{f}$ are diffeomorphic.

We can choose a base point $t'\in D_{\delta}-\{t_1,...,t_{\mu}\}$ and paths $p_i, 1\le i\le \mu$ connecting $t'$ to $t_i$ such that its interior is contained in $D_{\delta}-\{t_1,...,t_{\mu}\}$. We define a loop $l_i$ based at $t'$ where $l_i$ goes around $t_i$ anticlockwise along a small circle centered at $t_i$ and is connected by $p_i$. The loops $l_1,...,l_{\mu}$ generate the fundamental group $\pi_1(D_{\delta}-\{t_1,...,t_{\mu}\},t')$. The loop $l_i$ induces monodromy action on the cohomology of fiber $H^{n-1}(F,\mathbb Z)$ given by the Picard-Lefschetz formula
$$T_i:\alpha\mapsto \alpha+(\alpha,\delta_i)\delta_i,$$
where $\delta_i$ is the vanishing cycle associated to the critical value $t_i$. The set of all vanishing cycles $\{\delta_i\}_{i=1}^{\mu}$ generates $H^{n-1}(F,\mathbb Z)$.  When $n$ is odd, $(\delta_i,\delta_i)=\pm 2$, while when $n$ is even, $(\delta_i,\delta_i)=0$.

\begin{definition}\normalfont
We define \textit{monodromy group} of the Milnor fiber of $f$ to be the subgroup of $\textup{Aut}H^{n-1}(F,\mathbb Z)$ generated by $T_1,...,T_{\mu}$.
\end{definition}
The monodromy group is independent of the choice of perturbation function and the loops $l_1,...,l_{\mu}$. Moreover, in the case where $n=3$ and $f(x_1,x_2,x_3)=0$ has ADE singularity at the origin, the following result is well known.
\begin{lemma}  (\cite[p.99]{Arnold},  \cite{GZ})\label{GabrielovLem}
Vanishing cycles $\delta_1,...,\delta_{\mu}$ can be chosen to form a basis of the root system of the corresponding ADE type in $H^2(F,\mathbb Z)$. The monodromy group of $f$ is the Weyl group corresponding to the type of singularity.
\end{lemma}
These vanishing cycles are obtained by a sequence of conjugation operations of paths $\{p_i\}_{i=1}^{\mu}$. Such operations are called Gabrielov operations.

Now let $S_0$ be the cubic surface arising from a hyperplane section of $X$ with an affine chart defined by $f(x_1,x_2,x_3)=0$ with an isolated singularity at $(0,0,0)$ of ADE type. The next result will show that deforming $f$ in the family of hyperplane sections is the "same" as considering the Milnor fiber theory of $f$.

\begin{lemma}\label{analyMilLem}
Choose a linear 2-dimensional hyperplane sections family parameterized by $(\lambda,w)\in \mathbb C^2$ with $(0,0)$ corresponds to $f(x_1,x_2,x_3)=0$ with an ADE singularity, then there is an $\varepsilon>0$ such that for all $\lambda,w$ with $|\lambda|,|w|<\varepsilon$, an affine chart of the total family has analytic equation
\begin{equation}
    f_{\lambda}(x_1,x_2,x_3)+w=0,\label{analyMilEqu}
\end{equation}
where $f_{\lambda}(x_1,x_2,x_3)$ is the affine equation of the hyperplane section at $(\lambda,0)$.
\end{lemma}

\begin{proof}
$f(x_1,x_2,x_3)=0$ is an affine cubic surface with an isolated singularity at $(0,0,0)$ of ADE type. Using $x_1,x_2,x_3,w$ as affine coordinates, the cubic threefold $X$ has equation
$$F(x_1,x_2,x_3,w)=f(x_1,x_2,x_3)+wQ(x_1,x_2,x_3)+w^2L(x_1,x_2,x_3)+w^3\sigma,$$
where $Q, L,\sigma$ are polynomials of degree $2,1$, and $0$, respectively. 

Projecting to the $w$-coordinate, we obtain a pencil 
\begin{equation}
   \mathcal{X}\to \mathbb C \label{pencil} 
\end{equation}
of hyperplane sections $f_w(x_1,x_2,x_3)=F(x_1,x_2,x_3,w)$ of $X$ through $f(x_1,x_2,x_3)=0$. Since the cubic threefold is smooth, $Q(0,0,0)\neq 0$. Therefore, the equation of cubic threefold is 
$$F(x_1,x_2,x_3,w)=f(x_1,x_2,x_3)+wG(x_1,x_2,x_3,w)=0,$$
where $G(x_1,x_2,x_3,w)$ is quadratic and is non-vanishing in a small neighborhood $D$ of $0$. Therefore, by restricting to $D$ and setting $g=f/G$, we get a family
$$g(x_1,x_2,x_3,w)+w=0,$$
which is analytically equivalent to the family $(\ref{pencil})$ restricted to $D$. 

Now we choose a perturbation of $f$ in the hyperplane section family transversal to the $w$ direction. In other words, we choose a linear function $$l=ax_1+bx_2+cx_3$$
with $a,b,c\in \mathbb C$ being general, then 
$$f_{\lambda}(x_1,x_2,x_3)=F(x_1,x_2,x_3,\lambda l), ~\lambda\in \mathbb C$$
is a pencil of hyperplane sections through $f$. We consider the two-dimensional family spanned by $l$ and $w$. Then for $(l,w)\in \mathbb C^2$, the hyperplane section at $\lambda l+w$ is defined by
\begin{equation}
    f_{\lambda,w}=F(x_1,x_2,x_3,\lambda l+w)=f_{\lambda}(x_1,x_2,x_3)+wG(x_1,x_2,x_3,w+\lambda l)+\lambda lH(x,y,z,w),\label{perturb}
\end{equation}
where $H(x,y,z,w)=G(x,y,z,w+\lambda l)-G(x,y,z,\lambda l)=wL(x_1,x_2,x_3)+(2w\lambda l+w^2)\sigma$ is divisible by $w$. 

Therefore, denote $G'=G+\lambda lH/w$, we can express the two dimensional family $(\ref{perturb})$ as 
$$\frac{f_{\lambda}(x_1,x_2,x_3)}{G'}+w=0,$$
in a small neighborhood $D^2$ of origin. It is analytically equivalent to the family
$$f_{\lambda}(x_1,x_2,x_3)+w=0.$$
\end{proof}

\noindent\textit{Proof of Proposition \ref{LocMonGpProp}}.
Let $\Sigma_0$ be the discriminant locus $x_0$, namely the locus $\{t\in B|p^{-1}(t)\cap D_0~\textup{is singular}\}$. $\Sigma_0\subseteq X^{\vee}\cap B$ is an irreducible component (when $X_0$ has only one isolated singularity, they are the same).

Since the complement of the inclusion $B^{\textup{sm}}\subseteq B\setminus \Sigma_0$ has real codimension at least two, there is a surjection
$$\pi_1(B^{\textup{sm}},t')\twoheadrightarrow \pi_1(B\setminus \Sigma_0,t'),$$
where $t'$ is a fixed base point. Therefore, one reduces to the case where $S_0$ has only one singularity and $\Sigma_0= X^{\vee}\cap B$.

We choose a general line $\mathbb L$ in $(\mathbb P^4)^*$ through $t'$ such that $\mathbb L$ intersect $\Sigma_0$ transversely at smooth points, then $U=B^{\textup{sm}}\cap \mathbb L$ is an analytic open space. Moreover, by a local version of Zariski's theorem on fundamental groups on a Lefschetz pencil \cite[Theorem 3.22]{Voisin2}, there is a surjection
\begin{equation}
    \pi_1(U,t')\twoheadrightarrow \pi_1(B^{\textup{sm}}, t'). \label{pi1Surj}
\end{equation}

Therefore it suffices to show that the monodromy representations generated by the loops in the 1-dimensional open space $U$ is the entire Weyl group.

On the other hand, by Lemma \ref{analyMilLem}, the hyperplane sections parameterized by $U$ are analytically equivalent to the family
$$f'(x_1,x_2,x_3)+w=0,$$
where $f'$ is the defining equation of the hyperplane section at $t'$ and is a perturbation of $f$. Therefore, by Lemma \ref{GabrielovLem}, the monodromy group induced by $\pi_1(U,t')$ is the Weyl group corresponding to the type of $x_0$.
\qed

\subsection{Globalization}
So far, we have been discussing local monodromy group around a single singularity on a cubic surface $X_{t_0}=X\cap H_{t_0}$ with ADE singularities. We want to relate the local monodromy group $G_{t_0,x_0}$ on cohomology of Milnor fiber around a singularity $x_0$ of $X_{t_0}$ to the monodromy group $G_{t_0}$ on the cohomology of nearby smooth hyperplane section $X_t$.
\begin{proposition}\label{MilInjProp}
The inclusion $F\hookrightarrow X_t$ of Milnor fiber induces an injection on homology
\begin{equation}
    H_2(F,\mathbb Z)\to H_2(X_t,\mathbb Z)_{\textup{van}}. \label{MilFibInjEqn}
\end{equation}
\end{proposition}
\begin{proof}
This is due to Brieskorn's theory \cite{Brieskorn} and its globalization \cite{Artin} (also see \cite{KM}, Theorem 4.43). Using the same notations as we introduced at the beginning of this section, there exists a finite cover $B'\to B$, such that the base-changed total family admits simultaneous resolution in the category of algebraic spaces. In other words, there is a commutative diagram as follows.

\begin{figure}[h]
    \centering
\begin{tikzcd}
\mathcal{X}'\arrow[dr, "g"] \arrow[r,"f"] &  \mathcal{X}\times_{B}B' \arrow[d]\arrow[r] &\mathcal{X} \arrow[d]\\
  & B'\arrow[r] &B
\end{tikzcd}
\end{figure}
$\mathcal{X}'$ is a complex analytic manifold, $f$ is bimeromorphic, and $g$ is a proper holomorphic submersion. (The resolution is in general not algebraic since the local gluing data is only analytic.)

$\mathcal{X}'\to B'$ is diffeomorphic to the product $X_t\times B'$ by Ehresmann's theorem, so the Milnor fiber $F=X_t\cap D_0$ is diffeomorphic to an open set $U$ of the central fiber $g^{-1}(0)$. The argument reduces to show that the homology group induced by the inclusion $U\hookrightarrow g^{-1}(0)$ is injective.

$g^{-1}(0)$ is isomorphic to the minimal resolution $\tilde{X}_0$ of $X_0$. Denote $V$ the exceptional curve in $\tilde{X}_0$ over $x_0$. Then $V$ is a bunch of $\mu$ (-2)-curves, and corresponds to a connected sub-diagram of the Dynkin diagram of $\mathbb E_6$. Since the image of $U$ in $X_0$ is a neighborhood of $x_0$, $U$ is a regular neighborhood of $V$. So the induced map $H_2(U,\mathbb Z)\to H_2(\tilde{X}_0,\mathbb Z)$ is injective.
\end{proof}

\begin{remark}\normalfont
Proposition \ref{MilInjProp} is false for elliptic singularity, since the Milnor number of such a singularity is 8, while the vanishing homology on nearby $X_t$ has rank 6.
\end{remark}

\begin{corollary} \label{LocGloMonoGrpCor}
Via the inclusion \eqref{MilFibInjEqn}, $H_2(F,\mathbb Z)$ becomes an irreducible sub $\pi_1(B^{\textup{sm}},t')$-representation of $H_2(X_t,\mathbb Z)_{\textup{van}}$. It induces an inclusion $$G_{t_0,x_0}\hookrightarrow G_{t_0}$$ from the monodromy group $G_{t_0,x_0}$ on cohomology of Milnor fiber of singularity $x_0$ to the monodromy group $G_{t_0}$ around $X_{t_0}$ \eqref{eqn_LocMonodromyAction-General}.
\end{corollary}

Now we're ready to prove the key proposition in this section.

\noindent\textit{Proof of Proposition \ref{MonoActGloProp}}. Denote $x_1,...,x_k$ the singularities of $X_{t_0}$ with ADE type. 
Let $W_i$ be the Weyl group corresponding to the type of the singularity $x_i$, then $W(R_e)=W_1\times\cdots\times W_k.$ We'll show that the local monodromy group $G_{t_0}$ is isomorphic to $W_1\times\cdots\times W_k$ as well.

Let $D_i$ denote a small ball in $\mathcal{X}$ around $x_i$ such that $D_i\cap D_j=\emptyset$ for $i\neq j$. Let $\Sigma_i=\{t\in B|X_t\cap D_i~\textup{is singular}\}$ be the discriminant locus of hyperplane sections of $X$ associated to $x_i$. Then $\Sigma_i$ is an irreducible analytic divisor of $B$ and $X^{\vee}\cap B=\cup_{i}\Sigma_i(x_i)$. None of the $\Sigma_i$ contains $\Sigma_j$ for $i\neq j$, since otherwise, the locus will extend to a proper curve, contradicting the fact that the dual variety $X^{\vee}$ is smooth in codimension one, and that the smooth locus parameterizes the hyperplane section with one ordinary nodal singularity.

Fix a general point $t'\in B^{\textup{sm}}$. We take a general pencil $\mathbb L$ in $(\mathbb P^4)^*$ through $t'$ intersecting $X^{\vee}\cap B$ transversely along the smooth locus. So $\mathbb L$ intersects each $\Sigma_i$ transversely at points $t_i^j$, for $j=1,...,\mu_i$, where $\mu_i$ is the Milnor number of $x_i$. None of the $t_i^j$ coincides with $t_{i'}^{j'}$ unless $i=i'$ and $j=j'$. There is a vanishing cycle $\delta_i^j\in H^2(X_{t'},\mathbb Z)$ associate to $t_i^j$. The monodromy action $T_i^j$ induced by the simple loop around $t_i^j$ on the 72 roots is given by the Picard-Lefschetz formula $(\ref{PLeqn})$ associated to $\delta_i^j$. Moreover, via the surjectivity $$\pi_1(\mathbb L\cap B^{\textup{}}, t')\twoheadrightarrow \pi_1(B^{\textup{sm}},t'),$$
the monodromy group $G_{t_0}$ defined in \eqref{eqn_LocMonodromyAction-General} is generated by $T_i^j$, $i=1,...,k$, $j=1,...,\mu_i$.
By Proposition \ref{LocMonGpProp} and Corollary \ref{LocGloMonoGrpCor}, the subgroup generated by $T_i^1,\ldots, T_i^{\mu_i}$ is the Weyl group $W_i$, which is also the subgroup generated by the reflections corresponding to the exceptional curves over $x_i$. 

Finally, since $\delta_i^j$ can be represented by a topological 2-sphere contained in the neighborhood $D_i$ around $x_i$, the intersection number $$(\delta_i^j,\delta_{i'}^{j'})=0,~i\neq i'.$$

Therefore, the monodromy operators $T_i^{j}$ and $T_{i'}^{j'}$ commute for $i\neq i'$ by Picard-Lefschetz formula \eqref{PLeqn}. Therefore, the subgroup corresponding to the monodromy group of cohomology on Milnor fiber of $x_i$ commutes with the subgroup corresponding to that of $x_j$. It follows that the monodromy group of $X_0$ is the product $W_1\times\cdots \times W_k.$ \qed

\section{Extension of the Topological Abel-Jacobi Map}\label{Section_Cubic3fold_ExtAJ}

Recall that the topological Abel-Jacobi map $\Psi_{\textup{top}}:\mathcal{T}_v\to J(X)$ \eqref{Intro_eqn_TAJ_cubic3fold} is one-to-one onto an open subspace of the theta divisor. We are interested in whether $\Psi_{\textup{top}}$ can extend to the compactifications of $\mathcal{T}_v$. 
\begin{proposition}\label{Prop_nonExtTAJ}
The topological Abel-Jacobi map $\mathcal{T}_v\to J(X)$ extends to a morphism  $\bar{\mathcal{T}}_v\to J(X)$ if and only if the cubic threefold $X$ has no Eckardt point.
\end{proposition}
\begin{proof}

Note that $\mathcal{T}_v\to J(X)$ always extends to a regular morphism on $\textup{Bl}_0(\Theta)$ via
\begin{equation}
    \mathcal{T}_v\hookrightarrow\textup{Bl}_0(\Theta)\to \Theta\hookrightarrow J(X). \label{Bl Theta to J(X)}
\end{equation}

According to Theorem \ref{barT'BlowupThm}, the map $\textup{Bl}_0(\Theta)\to \bar{\mathcal{T}}_v$ contracts an elliptic curve $E_i$ corresponding to an Eckardt point on $X$. Since the morphism $(\ref{Bl Theta to J(X)})$ sends $E_i$ isomorphically onto its image in $J(X)$, the rational map $\bar{\mathcal{T}}_v\dashrightarrow J(X)$ extends to a regular morphism if and only in $X$ has no Eckardt point.

\end{proof}

So according to the proof, the boundary points where $\Psi_{\textup{top}}$ does not extend are exactly the points over $t_i\in (\mathbb P^4)^*$ such that $X_{t_i}=X\cap H_{t_i}$ is an Eckardt cone, namely a cone over an elliptic curve, and the cone point is an Eckardt point on $X$. For the same reason, the Abel-Jacobi map $\Psi$ does not extend to the Stein completion $ \overline{(F\times F)^{\circ}}$ when Eckardt point occurs.

\subsection{Semistable Reduction} We are in the situation described in Section \ref{Intro_Sec_ExtTAJ}. 

Let $X_0$ be a hyperplane section of $X$ with an elliptic singularity. Then $X_0$ is a cone over a plane cubic curve $E$.  Choose a general pencil of hyperplane sections through $X_0$ and restrict the family to a small holomorphic disk $\Delta$ with $t=0$ corresponding to $X_0$. Denote 
\begin{equation}\label{Eqn_pencilX0}
    \mathcal{X}\to \Delta
\end{equation}
 the corresponding pullback family of hyperplane sections of $X$.  
 
 Let $\bar{M}_{\textup{cubic}}$ be the moduli space of cubic surface arising from GIT \cite{ACT}. Then there is an open subspace $M_{\textup{cubic}}$ of $\bar{M}_{\textup{cubic}}$ parameterizing stable cubic surfaces, which are cubic surfaces with at worst $A_1$ singularities. The strict semistable locus consists of a single point, which corresponds to the cubic surface $xyz=w^3$ with three $A_2$ singularities.

In particular, the family \eqref{Eqn_pencilX0} is smooth over $\Delta^*=\Delta\setminus\{0\}$ and defines a rational map
\begin{equation*}
    m:\Delta\dashrightarrow \bar{M}_{\textup{cubic}},
\end{equation*}
whose regular extension $\bar{m}:\Delta\to \bar{M}_{\textup{cubic}}$ specifies a semistable limit $\bar{m}(0)\in \bar{M}_{\textup{cubic}}$.

To find the limiting cubic surface explicitly, we'll consider the semistable reduction of the family. We will show that after a base change by a cyclic cover $\tilde{\Delta}\to \Delta$ of order 3 totally branched at $0$, the family is birational to a smooth family. Let's denote $\tilde{\Delta}^*=\tilde{\Delta}\setminus\{0\}$.

\begin{proposition} \label{Prop_Semistable_reduction}
There is a smooth total space $\tilde{\mathcal{X}}$, and a flat family $f:\tilde{\mathcal{X}}\to \tilde{\Delta}$ extending the smooth family $\mathcal{X}\times_{\Delta^*}\tilde{\Delta}^*$. Moreover, 

\begin{itemize}
    \item[(1)] the special fiber $\tilde{\mathcal{X}}_0\cong \tilde{X}_0\cup Z$ has two irreducible components, where $\tilde{X}_0$ is the blow-up at the cone point,  $Z$ is the cubic surface arising as cyclic cover of $\mathbb P^2$ along a cubic curve $E$. Two components intersect transversely along $E$. 
    \item[(2)] The normal bundle $N_{\tilde{X}_0|\tilde{\mathcal{X}}}$ restricts to the ruling of $\tilde{X}_0$ is isomorphic to $\mathcal{O}(-1)$.
    \item[(3)] The ruling of $\tilde{X}_0$ is extremal in the relative Mori cone $\overline{NE}(\tilde{\mathcal{X}}/\tilde{\Delta})$.

\end{itemize}
\end{proposition}

\begin{proof}
(1) Blow up $\mathcal{X}$ at the cone point $p$ of $X_0$, then take the base change with respect to the 3:1 cover $\tilde{\Delta}\to \Delta$ and normalize, we get a family 
\begin{equation}
    \tilde{\mathcal{X}}\to \tilde{\Delta},\label{E_basechange3:1}
\end{equation}
whose fiber at 0 is $\tilde{X}_0\cup Z$, where $Z$ is the triple cover of $\mathbb P^2$ branched along $E$, $\tilde{X}_0$ is a ruled surface, and two components meet transversely along $E$. 

(2) Since $\mathcal{O}(\tilde{X}_0+Z)=\mathcal{O}_{\tilde{X}_0}$, one has $\mathcal{O}(\tilde{X}_0+Z)|_{\tilde{X}_0}=\mathcal{O}_{\tilde{X}_0}$. On the other hand, $\mathcal{O}(Z)|_{\tilde{X}_0}=\mathcal{O}_{\tilde{X}_0}(\tilde{X}_0\cap Z)=\mathcal{O}_{\tilde{X}_0}(E_{\infty})$, where $E_{\infty}$ is the section at infinity, i.e., the divisor with $E_{\infty}^2=-3$. Therefore
\begin{equation}\label{eqn_Cubic3fold_StableReduction6}
    N_{\tilde{X}_0|\tilde{\mathcal{X}}}=\mathcal{O}(\tilde{X}_0)|_{\tilde{X}_0}=\mathcal{O}_{\tilde{X}_0}(-E_{\infty}).
\end{equation}

Since $E_{\infty}$ intersects each member of the ruling transversely at one point, the normal bundle restricted to the ruling has degree $(-1)$.

$(3)$  Let $F$ be the ruling on $\tilde{X}_0$, then as divisors,
\begin{equation}\label{eqn_MoriCone}
    F=aE_{\infty}+bC,
\end{equation}
where $C$ is an effective curve whose irreducible components are not contained in the ruled surface $\tilde{X}_0$. Then by intersecting both sides of \eqref{eqn_MoriCone} with $\tilde{X}_0$ and use \eqref{eqn_Cubic3fold_StableReduction6}, one obtains 

$$-1=3a+bC\cdot \tilde{X}_0.$$
Since $C\cdot \tilde{X}_0\ge 0$, $a$ and $b$ cannot be both positive.

\end{proof}
\begin{remark}\normalfont
In Proposition \ref{Prop_Semistable_reduction}, one may base change first and then blow up the cone point to obtain the family \eqref{E_basechange3:1}.
\end{remark}

\begin{corollary}
    $f$ factors through the diagram
\begin{figure}[ht]
    \centering
\begin{equation}
\begin{tikzcd}
\tilde{\mathcal{X}}\arrow[r,"g"] \arrow[d,"f"] & W  \arrow[dl,"h"] \\
\tilde{\Delta},
\end{tikzcd}
\end{equation}
\end{figure}{}

\noindent where $W$ is a smooth projective variety, $g$ is birational and blows down the ruled surface $\tilde{X}_0$ to an elliptic curve $E$ in $W$, $h$ is a smooth morphism, and the fiber $h^{-1}(0)$ is isomorphic to $Z$.
\end{corollary}
\begin{proof}
    This follows from Proposition \ref{Prop_Semistable_reduction} and a relative version of Mori's Cone theorem \cite[Theorem 3.25]{KM}. The smoothness of $W$ follows from that the normal bundle restricting to ruling is $\mathcal{O}(-1)$ \cite[Theorem 3.2.8]{BS95}. 
\end{proof}

\begin{corollary}\label{Cor_StableLim}
    The semistable limit $\bar{m}(0)$ of a general pencil of hyperplane sections of $X$ through the cone $X_0$ over an elliptic curve $E$ is a smooth cubic surface $Z$ arising as the cyclic cover of $\mathbb P^2$ branched along a cubic curve isomorphic to $E$.
\end{corollary}

\begin{proposition}\label{Prop_monodromy_order_3}
The monodromy group of the family $\mathcal{X}^*\to \Delta^*$ is $\mathbb Z_3$. Moreover, it acts freely on the 27 lines and the 72 roots.
\end{proposition}
\begin{proof}
By Ehresmann's theorem, $W\to \tilde{\Delta}$ is topologically trivial and has trivial monodromy group. So the monodromy group of the family $\mathcal{X}^*\to \Delta^*$ is a subgroup of $\Z_3$.

Note that the generator of the monodromy cyclic permutes the three sheets of $Z\cong W_0$. On the other hand, the limiting 27 lines are preimages of the tangent lines to the 9 flex points of cubic curve $E$ under the 3-to-1 cover $Z\to \mathbb P^2$. So the monodromy permutes the three lines over each flex point. In particular, the monodromy group is not zero. So it has to be the entire $\Z_3$.

Finally, for each $1\le i\le 9$, let $L_{i1}, L_{i2}, L_{i3}$ denote the three lines over the tangent line of the $i$-th flex point of $E$. We can assume for each $j$, the lines $L_{ij}$ $1\le i\le 9$ lie in one sheet. Therefore, for a given line, say $L_{i1}$, the 10 others lines that intersects it are $L_{i2}$, $L_{i3}$ and $L_{i'1}$ for $i'\neq i$. In particular, $L_{i1}$ is disjoint from $L_{i'2}$ and the difference $[L_{i1}]-[L_{i'2}]$ is a root. The monodromy action on a root has the form 
$$[L_{i1}]-[L_{i'2}]\mapsto [L_{i2}]-[L_{i'3}].$$

Since $([L_{i1}]-[L_{i'2}])\cdot ([L_{i2}]-[L_{i'3}])=[L_{i1}]\cdot [L_{i2}]-[L_{i'2}]\cdot [L_{i2}]+[L_{i'2}]\cdot [L_{i'3}]=1-1+1=1$. In particular, the monodromy does not fix any root (whose self-intersection is $-2$). So the monodromy is free on the 72 roots.
\end{proof}

\subsection{Construction of New Completion}

Now let's construct a new compactification of $\mathcal{T}_v$: We blow up $(\mathbb P^4)^*$ at (finitely many) points corresponding to Eckardt hyperplanes. Denote the new space by $\tilde{(\mathbb P^4)^*}$. 

Applying Stein's completion lemma (cf. Lemma \ref{analytic_cover_closure}) to the finite covering $\pi_v:\mathcal{T}_v\to U$ with respect to the completion $U\subseteq \tilde{(\mathbb P^4)^*}$, we obtain a normal algebraic variety $\tilde{\mathcal{T}}_v$ together with a finite morphism $\tilde{\pi}_v:\tilde{\mathcal{T}}_v\to \tilde{(\mathbb P^4)^*}$ extending $\pi_v$.

Denote $P\cong \mathbb P^3$ a connected component of the exceptional divisor on $\tilde{(\mathbb P^4)^*}$, then a general $l\in P^3$ corresponds to a one-parameter family of hyperplanes \eqref{Eqn_pencilX0} through $X_0$, and the fiber $\tilde{\pi}^{-1}(l)$ should be the refined notion of "limiting primitive vanishing cycles" for Eckardt hyperplane section.
We have the following characterizations of the new completion $\tilde{\mathcal{T}}_v$.

\begin{proposition}\label{Prop_Tv-tilde-24}
(i) $\tilde{\mathcal{T}}_v$ is isomorphic to the normalization of $\bar{\mathcal{T}}_v\times_{(\mathbb P^4)^*}\tilde{(\mathbb P^4)^*}$.

(ii)  $\tilde{\pi}_P:\tilde{\mathcal{T}}_v\times_{\tilde{(\mathbb P^4)^*}}P\to P$ is finite and has degree $d=24$.
\end{proposition}
\begin{proof}
(i) Since $\bar{\mathcal{T}}_v\times_{(\mathbb P^4)^*}\tilde{(\mathbb P^4)^*}\to \tilde{(\mathbb P^4)^*}$ is also finite and extends $\mathcal{T}_v\to U$, by the uniqueness of Stein's completion, its normalization has to be isomorphic to $\tilde{\mathcal{T}}_v$.

(ii) Finiteness is stable under base change. A general point in $P$ corresponds to a one-parameter family \eqref{Eqn_pencilX0}. By Proposition  \ref{Prop_monodromy_order_3}, the monodromy group is $\Z_3$ and acts freely on the 72 roots. So there are $72/3=24$ orbits.
\end{proof}

From another point of view, if we specify a one-parameter family of hyperplane sections through the Eckardt hyperplane, the 27 lines on the nearby fiber specialize to 27 lines on the Eckardt cone, so does a vanishing cycle represented by the difference of the classes of two skew lines. This suggests that the topological Abel-Jacobi map extends along this one-dimensional disk.

\begin{remark}\normalfont
The branching locus of $\tilde{\pi}_P$ corresponds to the set of pencils that are "not general" in the sense that they are pencils passing through $X_0$ and are contained in the dual variety $X^*$. It is not hard to verify that for each $L$ in the ruling of $X_0$, the pencil defined by the plane $P_L$ which is tangent to $X$ along $L$ (cf. \cite[Lemma 6.7]{CG}) is contained in $X^*$. So the branching locus of $\tilde{\pi}_P$ has dimension at least one.
\end{remark}

\subsection{Extension of Topological Abel-Jacobi Map}

We will show that the topological Abel-Jacobi map $\mathcal{T}_v\to J(X)$ extends to the new compactification $\tilde{\mathcal{T}}_v$. In fact, we will show something stronger, namely the extension lifts to the blow-up $\Bl_0J(X)$.

\begin{proposition}\label{Prop_TAJExt}
There is a morphism $\tilde{\mathcal{T}}_v\to \Bl_0J(X)$ extending the topological Abel-Jacobi map $\mathcal{T}_v\to J(X)$.
\end{proposition}

\begin{proof}
According to Lemma \ref{lemma_ExtendGauss}, the morphism $\tilde{\mathcal{G}}: \textup{Bl}_{0}(\Theta)\to (\mathbb P^4)^*$ has fiber $E_i$ over an Eckardt hyperplane $t_i\in (\mathbb P^4)^*$. We denote $\tilde{\Theta}:=\textup{Bl}_{0}(\Theta)$. Then   $\tilde{\Theta}^{\circ}:=\tilde{\Theta}\setminus \cup_i(E_i)\to (\mathbb P^4)^*\setminus \{t_1,\ldots,t_k\}$ is a finite branched covering. Denote $\tau:\tilde{(\mathbb P^4)^*}\to (\mathbb P^4)^*$ the blowup at $t_i$. Then we can take the closure of the pullback of $\tilde{\Theta}^{\circ}$ in the fiber product 
\begin{equation}\label{eqn_ext-AJ-pullbackclosure}
  \overline{\tau^{-1}(\tilde{\Theta}^{\circ})}\subseteq \tilde{\Theta}\times_{(\mathbb P^4)^*}\tilde{(\mathbb P^4)^*}.  
\end{equation}

The projection to the second coordinate $\pi_2:\overline{\tau^{-1}(\tilde{\Theta}^{\circ})}\to \tilde{(\mathbb P^4)^*}$ is finite. In fact, let $l\in P_i$ corresponding to a pencil $\mathbb L_l\subseteq (\mathbb P^4)^*$ of hyperplanes $\{H_t=H_{t_i}+tH_l=0\}$, then the fiber $\pi_2^{-1}(l)$ is the limit of the finitely many points
$\tilde{\mathcal{G}}^{-1}(t)$ as $t$ goes to $0$. In other words, the closure of $(\mathbb L_l\setminus {0})\times_{(\mathbb P^4)^*}\tilde{\Theta}$ in $\tilde{\Theta}$ is a finite cover over an open neighborhood of $0\in L_l$ and its fiber over $t=0$ corresponds to $\pi_2^{-1}(l)$.

By Stein's Lemma \ref{analytic_cover_closure}, the normalization of $\overline{\tau^{-1}(\tilde{\Theta}^{\circ})}$ is $\tilde{\mathcal{T}}_v$. Now the argument follows from that the composite $\tilde{\mathcal{T}}_v\to \overline{\tau^{-1}(\tilde{\Theta}^{\circ})}\to \textup{Bl}_0(\Theta)\hookrightarrow \Bl_0J(X)$ extends the topological Abel-Jacobi map $\mathcal{T}_v\to J(X)$.
\end{proof}

\noindent \textit{Proof of Theorem \ref{Intro_Thm_TAJext}.}
It follows from Proposition \ref{Prop_monodromy_order_3}, \ref{Prop_Tv-tilde-24}, and \ref{Prop_TAJExt}.
\qed

\section{Tube Mapping}\label{Section_Cubic3fold_Schnell}

In \cite{Tube}, Schnell studied the relationship between the primitive homology $H_n(X,\mathbb Z)_{\textup{prim}}$ of a smooth projective variety $X\subseteq \mathbb P^N$ of dimension $n$ and the vanishing homology $H_{n-1}(Y,\mathbb Z)_{\textup{van}}$ of a smooth hyperplane section $Y=X\bigcap H$. Let $U\subseteq (\mathbb P^N)^*$ be the open set of smooth hyperplanes, and $l\subseteq U$ be a loop based at $t$, and $\alpha\in H_{n-1}(Y,\mathbb Z)_{\textup{van}}$, if $l_*\alpha=\alpha$, then the trace of the parallel transport of $\alpha$ along the loop $l$ is a topological $n$-chain on $X$ with boundary $\alpha-l_*\alpha=0$, so it is a $n$-cycle which is well-defined in the primitive homology. Since the $n$-cycle is a "tube" on $\alpha$ over the loop $l$, such map is called \textit{tube mapping}. Schnell proved that

\begin{theorem}(\cite{Tube})
If $H^{n-1}_{\textup{van}}(Y,\mathbb Z)\neq 0$, then the tube map
$$\{([l],\alpha)\in \pi_1(U,t)\times H_{n-1}(Y,\mathbb Z)_{\textup{van}}|l_*\alpha=\alpha\}\to H_n(X,\mathbb Z)_{\textup{prim}}$$
has a cofinite image. 
\end{theorem}


When the tube mapping is restricted to "tubes" over a single primitive vanishing cycle $\alpha_0$, it becomes
\begin{equation}
    \{([l],\alpha_0)\: |\:[l]\in \pi_1(U,t),l_*\alpha_0=\alpha_0\}\to H_n(X,\mathbb Z)_{\textup{prim}}
\end{equation}

Now Proposition \ref{Intro_Prop_Clemens_conj} will follow from the following two arguments.

\begin{lemma}
    Suppose $n$ is odd, then the map \eqref{tube mapping} agrees with the map between fundamental groups 
    \begin{equation}\label{eqn_map_on_pi1}
  \pi_1(\mathcal{T}_v,\alpha_0)\to \pi_1(J_{\prim}(X),0) 
\end{equation}
    induced by the topological Abel-Jacobi map \eqref{Intro_eqn_TAJ}.
\end{lemma}
The original argument is in \cite[p.26]{Zhao}. We provide self-contained proof here.
\begin{proof}
   First, every loop $l$ in $U$ based at $t$ fixing a primitive vanishing cycle $\alpha_0$ lifts to a loop $\tilde{l}$ in $\mathcal{T}_v$ based at $\alpha_0$, and vice versa, so the left-hand side of \eqref{tube mapping} is identified to the fundamental group $\pi_1(\mathcal{T}_v,\alpha_0)$.

   To show \eqref{eqn_map_on_pi1} agrees with \eqref{eqn_map_on_pi1}, let $\tilde{l}\subseteq \mathcal{T}_v$ be a loop based at $\alpha_0$. Let $[0,1]\to \tilde{l}$ be a parameterization. Then the image of $\tilde{l}$ under $\Psi_{\textup{top}}$ is determined by a family of $n$-chains $\Gamma_t$ indexed by $t\in [0,1]$ modulo $n$-cycles on $X$, so we can choose $\Gamma_t$ to be the union $\Gamma_0\bigcup\Gamma_t'$ where $\Gamma_t'=\bigcup_{s\in[0,t]}\alpha_s$ as trace of primitive vanishing cycles along the path $[0,1]$. It follows that $\Gamma_1$ is a $n$-chain such that $\partial \Gamma_1=\partial \Gamma_0=\alpha_0$, so the induced map on $\pi_1$ sends $\tilde{l}$ to the image of the $n$-cycle $\Gamma_1-\Gamma_0=\bigcup_{t\in[0,1]}\alpha_t$ in $H_n(X,\mathbb Z)$.
\end{proof}

\begin{proposition}\label{Prop_TubeCubic3fold}
When $X$ is a smooth cubic threefold, the map (\ref{eqn_map_on_pi1}) is surjective.
\end{proposition}
\begin{proof}
First of all, $\phi: \mathcal{T}_v\to J(X)$ factors through the inclusion $\mathcal{T}_v\subseteq \textup{Bl}_0(\Theta)$. Moreover, $\mathcal{T}_v\subseteq \textup{Bl}_0(\Theta)$ is a complement of a divisor in a smooth complex manifold, as a smooth loop based can be deformed to be disjoint from a real codimension-two set, there is a surjection $\pi_1(\mathcal{T}_v,*)\twoheadrightarrow \pi_1(\textup{Bl}_0(\Theta),*)$. Therefore, it suffices to show that $\pi_1(\textup{Bl}_0(\Theta),*)\to \pi_1(J(X))$ is surjective.

Next, choose $p\in F$ such that its corresponding line $L_p$ is of the second type on $X$ and let $D_p$ be the divisor of lines that are incident to $L_p$. By Lemma 10.7 of \cite{CG}, $p\in D_p$, it follows that $\{p\}\times F\setminus D_p$ is disjoint from the diagonal. In particular, let $\sigma: \textup{Bl}_{\Delta_F}(F\times F)\to F\times F$ be the blowup map, the restriction of $\sigma^{-1}$ to the domain of $\Psi_p$ is an isomorphism. We define the restricted Abel-Jacobi map
\begin{equation}
   \Psi_p:\{p\}\times F\setminus D_p\to J(X). \label{open AJ}
\end{equation}

$\Psi_p$ lifts to the blowup, so the image of  $\pi_1(\textup{Bl}_0(\Theta),*)\to \pi_1(J(X))$ contains $(\Psi_{p})_*(\pi_1(\{p\}\times F\setminus D_p,*))$ as a subgroup. Thus it suffices to show that $\Psi_p$ induces surjectivity on fundamental groups. 

To show this, note that $\Psi_p$ factors through the inclusion $\{p\}\times F\setminus D_p\subseteq \{p\}\times F$, which induces a surjective map on the fundamental group for the same reason as in the first paragraph of the proof.
Moreover, the map $\{p\}\times F\cong F\to J(X)$ factors through the Albanese map 

\begin{figure}[ht]
    \centering
\begin{equation}\label{graphiso}
\begin{tikzcd}
F\arrow[r,"\Psi"] \arrow[d,"alb"] & J(X)   \\
\textup{Alb}(F)\arrow[ur,"\cong"]
\end{tikzcd}
\end{equation}
\end{figure}{}
\noindent together with the isomorphism $\textup{Alb}(F)\xrightarrow{\cong} J(X)$ \cite{CG}. It follows that $\Psi$ induces an isomorphism between fundamental groups. Therefore, so does $\Psi_p$. Note that $H_3(X,\mathbb Z)=H_3(X,\mathbb Z)_{\textup{prim}}$ for cubic threefold, so the result follows.
\end{proof}

\appendix

\section{Primitive Vanishing Cycles}\label{App_PVC}
Let $X\subseteq \mathbb P^N$ be a smooth projective variety of dimension $n$. Let $U$ be the open subspace of $(\mathbb P^N)^*$ parameterizing smooth hyperplane sections of $X$. 

According to a classical result by Lefschetz, a smooth point of $X^*=(\mathbb P^N)^*\setminus U$ corresponds to a hyperplane section that has only one ordinary node. Choose a line $\mathbb L\subset (\mathbb P^N)^*$ intersecting transversely to a smooth point on $X^*$. Take a holomorphic disk $\Delta$ on $\mathbb L$ centered at that point. Then $\{X_t\}_{t\in \Delta}$ is a one-parameter family of hyperplane sections of $X$ with $X_0$ having a single node and $X_t$ smooth for $t\neq 0$. Let $\mathcal{X}_{\Delta}$ denote the total space, and $B_p\subseteq \mathcal{X}_{\Delta}$ a small neighborhood of the node $p\in X_0$. When $|t|$ is small enough, the manifold $X_t\bigcap B_{p}$ is called the Milnor fiber. It deformation retracts to a topological $(n-1)$-sphere $S^{n-1}$. Moreover, the sphere $S^{n-1}$ specializes to the node $p$ as $t$ moves to $0$. As a result, the homology class of $S^{n-1}$ is zero in homology of $X$ and defines an element in the vanishing homology $H_{n-1}(X_t,\Z)_{\textup{van}}:=\ker (H_{n-1}(X_t,\Z)\to H_{n-1}(X,\Z))$.

By Poincar\'e duality, the class $[S^{n-1}]$ lies in the \textit{vanishing cohomology} defined as the kernel of Gysin homomorphism
\begin{equation}\label{Eqn_VanishingCoh}
    H^{n-1}_{\textup{van}}(X_t,\mathbb Z):=\ker(H^{n-1}(X_t,\mathbb Z)\to H^{n+1}(X,\mathbb Z)).
\end{equation}

\begin{definition}\normalfont
The cohomology class $\delta=[S^{n-1}]\in \Hv^{n-1}(X_t,\mathbb Z)$ is called the \textit{vanishing cycle} of the nodal degeneration $\{X_t\}_{t\in \Delta}$. 
\end{definition}

Let $\mathcal{H}_{\textup{van}}^{n-1}$ be the local system on $U$ whose stalk at $t$ is the vanishing cohomology $H^{n-1}_{\textup{van}}(X_t,\Z)$. Let $T$ denote the \'etale space of $\mathcal{H}_{\textup{van}}^{n-1}$, then 
\begin{equation}
  T\to U  \label{eqn_T-to-Osm}
\end{equation}
is an analytic covering space. Note that $T$ has possibly infinitely many connected components. For example, when $n$ is odd, $\alpha$ and $2\alpha$ lie in different components since they have different self-intersections.

\begin{proposition} \label{Prop_Tv}
$T$ has a unique connected component $\mathcal{T}_v$ that contains the vanishing cycle of nodal degeneration.
\end{proposition}
\begin{proof}
This follows from the fact that the dual variety $X^*$ is irreducible and all vanishing cycles of nodal degenerations $\{X_t\}_{t\in \Delta}$ obtained from above are conjugate to each other \cite[Proposition 3.23]{Voisin2}.
\end{proof}
\begin{definition}\label{Def_Tv}
\normalfont We call $\mathcal{T}_v$ the primitive vanishing cycle component on the hyperplane sections of $X$. We call $\alpha_t\in \Hv^{n-1}(X_t,\Z)$ a \textit{primitive vanishing cycle} if $\alpha_t$ lies in $\mathcal{T}_v$.
\end{definition}

Equivalently, Let $t'\in U$ be a point close to $t_0\in X^*$ where $X_{t_0}$ has an ordinary node. Let $\delta$ be a vanishing cycle associated to the nodal degeneration as $X_{t'}$ approaches $X_{t_0}$, then a class $\alpha\in H^{n-1}_{van}(X_t,\mathbb Z)$ is a \text{primitive vanishing cycle} if there exists a smooth path $l\subseteq U$ joining $t$ to a point $t'$ and the monodromy image $l_*(\delta)=\alpha$.

\begin{definition}
\normalfont Denote $PV_t$ the fiber of $\pi_v:\mathcal{T}_v\to U$ over $t$. Call $PV_t$ the set of primitive vanishing cycles on the hyperplane section $X_t$.
\end{definition}

By definition, the set $PV_t$ of all primitive vanishing cycles on a smooth hyperplane section $X_t$ is the orbit of a single vanishing cycle $\alpha_t$ under the monodromy action
\begin{equation}\label{eqn_GlobalMonodromy-General}
    \rho_{\textup{van}}:\pi_1(U,t)\to \textup{Aut}H^{n-1}_{\textup{van}}(X_{t},\Z).
\end{equation}

\begin{proposition}\label{Prop_PVC-generation}
The set of all primitive vanishing cycles in $H^{n-1}_{\textup{van}}(X_{t},\Z)$ generates a sublattice of full rank.
\end{proposition}
\begin{proof}
It is well known that the vanishing cycles in a Lefschetz pencil generate the $H^{n-1}_{\textup{van}}(X_{t},\mathbb Q)$ \cite[Lemma 2.26]{Voisin2}. These vanishing cycles are a subcollection of primitive vanishing cycles.
\end{proof}

\section{Compactification of Local System}\label{App_Schnell}

\subsection{Schnell's Completion}
Suppose $(\mathcal{H},Q)$ is a polarized variation of Hodge structure of even weight over a quasi-projective variety $B_0$. Suppose $B_0$ is a Zariski open subset of a smooth projective variety $B$, Schnell \cite{SchExtLocus} constructed a completion of the space $T_{\Z}$, the \'etale space of the local system $\mathcal{H}_{\mathbb Z}$.

More explicitly, assume that $\mathcal{H}$ has weight $2n$. The data $(\mathcal{H},Q)$ consists of a $\mathbb Z$-local system over $B_0$, a flat connection $\nabla$ on $\mathcal{H}_{\mathbb C}=\mathcal{H}\times_{\mathbb Z}\mathcal{O}_{B_0}$, Hodge bundles $F^p\mathcal{H}_{\mathbb C}$ and a nondegenerate pairing 
$$Q:H_{\mathbb Q}\times H_{\mathbb Q}\to \mathbb Q$$
satisfying the Hodge-Riemann conditions.

Consider $F^n\mathcal{H}$ the associated Hodge bundle, i.e., the subbundle whose fiber at $p\in B_0$ is $F^n\mathcal{H}_p$, the $n$-th Hodge filtration of the complex vector space $\mathcal{H}_p$. Then it is shown in Lemma 3.1 from \cite{SchExtLocus} that for each connected component $T_{\lambda}/B_0$ of $T_{\mathbb Z}/B_0$, the natural mapping 
$$T_{\lambda}\to T(F^n\mathcal{H})$$
$$\alpha\mapsto Q(\alpha,\cdot)$$
is finite, where $T(F^n\mathcal{H})$ is the underlying analytic space of the Hodge bundle. 

Moreover, according to Saito's Mixed Hodge Modules theory, there is a Hodge module $M$ underlying a filtered $\mathcal{D}_{B_0}$-module $(\mathcal{M},F_{\bullet}\mathcal{M})$ supported on $B$, as the minimal extension of $(\mathcal{H},\nabla)$. 

Schnell considered the space $T(F_{n-1}\mathcal{M})$ as the analytic spectrum of the $(n-1)$-th filtration of $\mathcal{M}$ and showed that the analytic closure of the image of the composite of 
$$\varepsilon: T_{\lambda}\to T(F^n\mathcal{H})\to T(F_{n-1}\mathcal{M})$$
is still analytic. Therefore, it extends to a finite analytic covering by Grauert's theorem, so there is a normal analytic space $\bar{T}_{\lambda}$ extending $T_{\lambda}$.

\begin{lemma}\cite[Theorem 4.2, 23.1]{SchExtLocus} \label{Lemma_Schnell_ext}
 There is a normal holomorphically convex analytic space $\bar{T}_{\lambda}$ containing $T_{\lambda}$ as an open dense subspace, and a finite holomorphic mapping 
 $$\bar{\varepsilon}:\bar{T}_{\lambda}\to T(F_{n-1}\mathcal{M})$$
with discrete fibers that extend $\varepsilon$.

\end{lemma}

Schnell defines $\bar{T}_{\mathbb Z}$ as the union $\bigcup_{\lambda}\bar{T}_{\lambda}$. The closed analytic subscheme $\bar{\varepsilon}^{-1}(0)\subseteq \bar{T}_{\Z}$ is defined to be the \textit{extended locus of Hodge classes}.

When the variation of the Hodge structure comes from the vanishing cohomology on the universal smooth hyperplane sections of a smooth hypersurface in projective space, the minimal extension $\mathcal{M}$ can be described as Griffiths' residues.

\subsection{Finite Monodromy}
When $T$ parameterizes only Hodge classes, the Hodge bundle is trivial, and $T\to B_0$ has finite monodromy. Schnell's compactification (cf. Lemma \ref{Lemma_Schnell_ext}) becomes the compactification due to Stein \cite{Stein} and Grauert-Remmert \cite{GR}. Also see \cite[p.197]{DG}.

\begin{lemma}\label{analytic_cover_closure}
Let $U$ be a complex manifold and $f:W\to U$ a finite analytic cover. Assume $\bar{U}$ is a normal analytic space containing $U$ as an open dense subspace. Then there is a normal analytic space $\bar{W}$ containing $W$ as a dense open subspace, together with finite analytic branched covering map $\bar{f}:\bar{W}\to \bar{U}$, which agrees with $f$ on $W$. Moreover, when $\bar{U}$ is projective, $\bar{W}$ is also projective.
\end{lemma}

We will only give an account for the algebraicity argument. The pushforward $\mathcal{F}=\bar{f}_*\mathcal{O}_W$  defines an analytic coherent sheaf on $\bar{U}$. By Serre's GAGA, the projectivity of $\bar{U}$ implies that $\mathcal{F}$ is an algebraic coherent sheaf. Then by definition of the relative spec construction \cite[Exercise II.5.17]{Hartshorne}, $\bar{W}$ is isomorphic to $\textup{Spec}_{\mathcal{O}_{\bar{W}}}{F}$, and therefore is algebraic.

\subsection{Infinite Monodromy}
When the variation of Hodge structure comes from hyperplane sections of a smooth cubic threefold, the monodromy is finite. However, this is very rare in general.

\begin{lemma}
Assume $X\subseteq \mathbb P^{n+1}$ is a smooth hypersurface of odd dimension. Let $U\subseteq (\mathbb P^{n+1})^*$ be the open subspace parameterizing smooth hyperplane sections. Then $T\to U$ has finite monodromy if and only if the vanishing cohomology $H^{n-1}_{\textup{van}}(X_t,\Z)$ is concentrated on Hodge type.
\end{lemma}
\begin{proof}
The sufficiency is straightforward since the intersection pairing is definite on the subspace $H^{\frac{n-1}{2},\frac{n-1}{2}}(X_t,\C)$. The necessity can be found in \cite[p.295]{Terasoma}. 
\end{proof}

\begin{corollary}
When $X$ is a hypersurface of $\mathbb P^4$ with degree at least $4$, $T\to U$ has infinite global monodromy and $$\mathcal{T}_v\to U$$ is a covering space of infinite sheets.
\end{corollary}

In fact, there is a quartic threefold with a hyperplane section having a triple point singularity, and the local monodromy around such a hyperplane section is infinite.

For cubic threefold, the topological Abel-Jacobi map is induced from the Abel-Jacobi map  \eqref{Intro_eqn_TAJ_cubic3fold}. We also characterized the compactification $\bar{\mathcal{T}}_v$ (cf. Section \ref{Section_Tvbar}) and explored the extension problem of the topological Abel-Jacobi map (cf. Section \ref{Section_Cubic3fold_ExtAJ}). We want to ask the same question for higher-degree hypersurfaces.
\begin{question}
How to describe the topological Abel-Jacobi map for hypersurface of $\mathbb P^4$ with degree at least 4? How to characterize Schnell's completion $\bar{\mathcal{T}}_v$?
\end{question}

We hope to study this problem beginning in quartic threefold in the future. For the second question, we may use \cite{ResidueDmodule}, where the minimal extension of the VHS is characterized using Griffiths residue.

\bibliographystyle{plain}
\bibliography{bibfile}

\end{document}